\documentclass[a4paper]{amsart}

\usepackage{enumerate, amsmath, amsfonts, amssymb, amsthm, wasysym, graphics, graphicx, xcolor, url, hyperref, hypcap, a4wide, stmaryrd, shuffle, xargs, multicol, multirow, overpic, pdflscape, hvfloat, minibox, accents, pifont, etoolbox, colortbl, calc, xparse}
\hypersetup{colorlinks=true, citecolor=darkblue, linkcolor=darkblue, urlcolor=darkblue}
\usepackage[all]{xy}
\usepackage[bottom]{footmisc}
\usepackage{tikz}\usetikzlibrary{trees,snakes,shapes,arrows,matrix,calc,fit}
\graphicspath{{figures/}{figures/relations/}}
\makeatletter
\def\input@path{{figures/}}
\makeatother

\newtheorem{theorem}{Theorem}
\newtheorem{corollary}[theorem]{Corollary}
\newtheorem{proposition}[theorem]{Proposition}
\newtheorem{lemma}[theorem]{Lemma}
\newtheorem{definition}[theorem]{Definition}
\newtheorem{conjecture}[theorem]{Conjecture}

\theoremstyle{definition}
\newtheorem{example}[theorem]{Example}
\newtheorem{remark}[theorem]{Remark}

\newcommand{\R}{\mathbb{R}} 
\newcommand{\N}{\mathbb{N}} 
\newcommand{\Z}{\mathbb{Z}} 
\renewcommand{\b}[1]{\mathbf{#1}} 
\newcommand{\fS}{\mathfrak{S}} 

\newcommand{\set}[2]{\left\{ #1 \;\middle|\; #2 \right\}} 
\newcommand{\bigset}[2]{\big\{ #1 \;\big|\; #2 \big\}} 
\newcommand{\ssm}{\smallsetminus} 
\newcommand{\dotprod}[2]{\langle #1 \, | \, #2 \rangle} 
\newcommand{\eqdef}{\mbox{\,\raisebox{0.2ex}{\scriptsize\ensuremath{\mathrm:}}\ensuremath{=}\,}} 
\renewcommand{\implies}{\;\Longrightarrow\;} 

\newcommandx{\rel}[1][1=R]{\mathbin{\mathrm{#1}}} 
\newcommandx{\notrel}[1][1=R]{\mathbin{\!\raisebox{.02cm}{$\not$}\hspace{.02cm}\mathrm{#1}\hspace*{.01cm}}} 

\newcommand{\IRel}{\mathcal{R}} 
\newcommand{\ISym}{\mathcal{S}} 
\newcommand{\IAntisym}{\mathcal{A}} 
\newcommand{\IClo}{\mathcal{C}} 
\newcommand{\ISemiClo}{\mathcal{SC}} 
\newcommand{\IPos}{\mathcal{P}} 

\newcommand{\WOEP}{\mathsf{WOEP}} 
\newcommand{\WOFP}{\mathsf{WOFP}} 
\newcommand{\WOIP}{\mathsf{WOIP}} 
\newcommandx{\CO}{\mathsf{CO}} 
\newcommandx{\COEP}{\mathsf{COEP}} 
\newcommandx{\COFP}{\mathsf{COFP}} 
\newcommandx{\COIP}{\mathsf{COIP}} 
\newcommand{\BOEP}{\mathsf{BOEP}} 
\newcommand{\BOFP}{\mathsf{BOFP}} 
\newcommand{\BOIP}{\mathsf{BOIP}} 

\newcommandx{\Inc}[1]{#1^{\mathsf{Inc}}} 
\newcommandx{\Dec}[1]{#1^{\mathsf{Dec}}}
\newcommand{\cl}[1]{#1^{\mathsf{cl}}} 
\newcommand{\pcd}[1]{#1^{\mathsf{pcd}}} 
\newcommand{\ncd}[1]{#1^{\mathsf{ncd}}} 
\newcommand{\extensions}{\mathcal{E}} 
\newcommand{\linearExtensions}{\mathcal{L}} 

\newcommand{\meet}{\wedge} 
\newcommand{\join}{\vee} 
\newcommand{\meetC}{\wedge_\IClo} 
\newcommand{\joinC}{\vee_\IClo} 
\newcommand{\meetSC}{\wedge_\ISemiClo} 
\newcommand{\joinSC}{\vee_\ISemiClo} 
\newcommand{\meetR}{\wedge_\IRel} 
\newcommand{\joinR}{\vee_\IRel} 
\newcommand{\joinWO}{\vee_W} 
\newcommand{\meetWO}{\wedge_W} 
\newcommand{\joinFWO}{\vee_{FW}} 
\newcommand{\meetFWO}{\wedge_{FW}} 
\newcommand{\joinWOEP}{\vee_\WOEP} 
\newcommand{\meetWOEP}{\wedge_\WOEP} 
\newcommand{\joinWOIP}{\vee_\WOIP} 
\newcommand{\meetWOIP}{\wedge_\WOIP} 
\newcommand{\joinWOFP}{\vee_\WOFP} 
\newcommand{\meetWOFP}{\wedge_\WOFP} 
\newcommand{\joinCOEP}{\vee_{\COEP(c)}} 
\newcommand{\meetCOEP}{\wedge_{\COEP(c)}} 
\newcommand{\joinCOIP}{\vee_{\COIP(c)}} 
\newcommand{\meetCOIP}{\wedge_{\COIP(c)}} 
\newcommand{\joinBOEP}{\vee_\BOEP} 
\newcommand{\meetBOEP}{\wedge_\BOEP} 
\newcommand{\joinBOIP}{\vee_\BOIP} 
\newcommand{\meetBOIP}{\wedge_\BOIP} 

\DeclareMathOperator{\inv}{inv} 
\DeclareMathOperator{\des}{des} 
\newcommand{\wole}{\preccurlyeq} 
\newcommand{\projDown}{\pi_\downarrow} 
\newcommand{\projUp}{\pi^\uparrow} 
\newcommand{\ProjDown}{\Pi_\downarrow} 
\newcommand{\ProjUp}{\Pi^\uparrow} 

\DeclareMathOperator{\cone}{cone} 
\newcommandx{\Asso}{\mathsf{Asso}} 
\newcommandx{\Perm}{\mathsf{Perm}} 
\newcommandx{\Para}[1][1=n]{\mathsf{Para}(#1)} 
\newcommand{\face}{\mathrm{F}} 

\newcommand{\tinymath}[1]{\scalebox{.5}{$#1$}} 
\newcommand{\tinymeet}{{\tinymath{\meet}}} 
\newcommand{\tinyjoin}{{\tinymath{\join}}} 
\newcommand{\zm}{z_\tinymeet} 
\newcommand{\Km}{{K_\tinymeet}} 
\newcommand{\zj}{z_\tinyjoin} 
\newcommand{\Kj}{{K_\tinyjoin}} 

\newcommand{\fref}[1]{Figure~\ref{#1}} 
\newcommand{\ie}{\textit{i.e.}~} 
\definecolor{darkblue}{rgb}{0,0,0.7} 
\definecolor{green}{RGB}{57,181,74} 
\newcommand{\darkblue}{\color{darkblue}} 
\newcommand{\defn}[1]{\emph{\darkblue #1}} 
\usepackage{todonotes}

\title{The weak order on Weyl posets}

\thanks{VP~was partially supported by the French ANR grant SC3A~(15\,CE40\,0004\,01).}

\author{Jo\"el Gay}
\author{Vincent Pilaud}

\address[JG]{LRI, Univ.\,Paris-Sud \& LIX, \'Ecole Polytechnique, Palaiseau}
\email{joel.gay@lri.fr}

\address[VP]{CNRS \& LIX, \'Ecole Polytechnique, Palaiseau}
\email{vincent.pilaud@lix.polytechnique.fr}
\urladdr{\url{http://www.lix.polytechnique.fr/~pilaud/}}

\begin{document}

\begin{abstract}
We define a natural lattice structure on all subsets of a finite root system that extends the weak order on the elements of the corresponding Coxeter group.
For crystallographic root systems, we show that the subposet of this lattice induced by antisymmetric closed subsets of roots is again a lattice.
We then study further subposets of this lattice which naturally correspond to the elements, the intervals and the faces of the permutahedron and the generalized associahedra of the corresponding Weyl group.
These results extend to arbitrary finite crystallographic root systems the recent results of G.~Chatel, V.~Pilaud and V.~Pons on the weak order on posets and its induced subposets.
\end{abstract}

\vspace*{-.9cm}

\maketitle

\vspace{-.4cm}

\section{Introduction}

The weak order is a fundamental ordering of the elements of a Coxeter group.
It can be defined as the prefix order in reduced expressions of the elements of the group, or more geometrically as the inclusion poset of the inversion sets of the elements of the group.
For finite Coxeter groups, the weak order is known to be a lattice~\cite{Bjorner} and its Hasse diagram is the graph of the permutahedron of the group oriented in a linear direction.
The rich theory of congruences of the weak order~\cite{Reading-latticeCongruences} yield to the construction of Cambrian lattices~\cite{Reading-cambrianLattices} with its  connection to Coxeter Catalan combinatorics and finite type cluster algebras~\cite{FominZelevinsky-ClusterAlgebrasI, FominZelevinsky-ClusterAlgebrasII}.
This point of view was fundamental for the construction of generalized associahedra~\cite{HohlwegLangeThomas}.
We refer to the survey papers~\cite{Reading-survey, Reading-FiniteCoxeterGroupsChapter, Hohlweg} for details on these subjects.

More recently, some efforts were devoted to develop certain extensions of the weak order beyond the elements of the group.
This led in particular to the notion of facial weak order of a finite Coxeter group, pioneered in type~$A$ in~\cite{KrobLatapyNovelliPhanSchwer}, defined for arbitrary finite Coxeter groups in~\cite{PalaciosRonco}, and proved to be a lattice in~\cite{DermenjianHohlwegPilaud}.
This order is a lattice on the faces of the permutahedron that extends the weak order on the vertices.

In type~$A$, an even more general notion of weak order on integer binary relations was recently introduced in~\cite{ChatelPilaudPons}.
This order is defined by~\mbox{$\rel[R] \wole \rel[S] \iff \Inc{\rel[R]} \supseteq \Inc{\rel[S]}$ and $\Dec{\rel[R]} \subseteq \Dec{\rel[S]}$} for any two binary relations~$\rel[R], \rel[S]$ on~$[n]$, where~$\Inc{\rel} \eqdef \set{(a,b) \in \rel}{a < b}$ and~$\Dec{\rel} \eqdef \set{(b,a) \in \rel}{a < b}$ respectively denote the increasing and decreasing subrelations of~$\rel[R]$.
It turns out that the subposet of this weak order induced by posets on~$[n]$ is a lattice.
In fact, many relevant lattices can be recovered as subposets of the weak order on posets induced by certain families of posets.
Such families include the vertices, the intervals and the faces of the permutahedron, associahedra~\cite{Loday, HohlwegLange}, permutreehedra~\cite{PilaudPons-permutrees}, cube, etc.
For the vertices, the corresponding lattices are the weak order on permutations, the Tamari lattice on binary trees, the type~$A$ Cambrian lattices, the permutree lattices~\cite{PilaudPons-permutrees}, the boolean lattice on binary sequences,~etc.

The goal of this paper is to extend these results beyond type~$A$.
We define the \defn{weak order} on subsets~$\rel[R], \rel[S]$ of a finite root system~$\Phi$ by \mbox{$\rel[R] \wole \rel[S] \iff \rel[R]^+ \supseteq \rel[S]^+$ and~$\rel[R]^- \subseteq \rel[S]^-$}, where~$\rel^+ \eqdef {\rel} \cap \Phi^+$ and~$\rel^- \eqdef {\rel} \cap \Phi^-$.
This order is a lattice on all subsets of~$\Phi$, which are the analogues of type~$A$ integer binary relations.
In turn, the analogues of type~$A$ integer posets are \defn{$\Phi$-posets}, \ie subsets~$\rel$ of~$\Phi$ that are both antisymmetric ($\alpha \in \rel$ implies~$-\alpha \notin \rel$) and closed (in the sense of~\cite{Bourbaki}, $\alpha, \beta \in \rel$ and~$\alpha + \beta \in \Phi$ implies $\alpha + \beta \in \rel$).
Our central result is that the subposet of the weak order induced by $\Phi$-posets is also a lattice when the root system~$\Phi$ is crystallographic.
For example, the weak orders on $A_2$-, $B_2$- and $G_2$-posets are represented in Figures~\ref{fig:posetsAB} and~\ref{fig:posetsG}.
Surprisingly, this property fails for non-crystallographic root systems, and the proof actually requires to develop delicate properties on subsums of roots in crystallographic~root~systems.

We then switch to our motivation to study the weak order on $\Phi$-posets.
We consider $\Phi$-posets corresponding to the vertices, the intervals and the faces of the permutahedron, the associahedra, and the cube of type~$\Phi$.
Considering the subposets of the weak order induced by these specific families of $\Phi$-posets allows us to recover the classical weak order and the Cambrian lattices, their interval lattices, and their facial lattices.

\section{Root systems}
\label{sec:rootSystems}

This section gathers some notions and properties on finite crystallographic root systems and Weyl groups.
We refer to the textbooks by J. Humphreys~\cite{Humphreys}, N.~Bourbaki~\cite{Bourbaki}, and A. Bj\"orner and F. Brenti~\cite{BjornerBrenti} for further details on basic definitions and classical properties.

\subsection{Root systems}

Let~$V$ be a real Euclidean space with scalar product~$\dotprod{\cdot}{\cdot}$.
For~$\alpha \in V \ssm \{0\}$, we define~$\alpha^\vee \eqdef 2 \alpha / \dotprod{\alpha}{\alpha}$.
We denote by~$s_\alpha$ the reflection orthogonal to a non-zero vector~$\alpha \in V$, defined by~$s_\alpha(v) = v - \dotprod{\alpha^\vee}{v} \, \alpha$.
A \defn{finite root system}~$\Phi$ is a finite set of non-zero vectors in~$V$ such that $\Phi \cap \R\alpha = \{\alpha, -\alpha\}$ and~$s_\alpha \Phi = \Phi$ for all~$\alpha \in \Phi$.
We denote by~$W$ the \defn{Coxeter group} generated by the reflections~$s_\alpha$ for~$\alpha \in \Phi$.
Throughout this paper, we will denote by~$\IRel(\Phi)$ the collection of all subsets of~$\Phi$.

We choose a generic linear functional~$f$ and denote by~$\Phi^+ \eqdef \set{\alpha \in \Phi}{f(\alpha) > 0}$ the set of \defn{positive roots} and by~$\Phi^- \eqdef \set{\alpha \in \Phi}{f(\alpha) < 0}$ the set of \defn{negative roots}.
We denote by~$\Delta$ the \defn{simple roots}.
They are the roots of the rays of the cone~$\R_{\ge 0} \Phi^+$ and form a linear basis, so that any positive root is a positive linear combination of simple roots.
The \defn{height} of a root~$\alpha = \sum_{\delta \in \Delta} \alpha_\delta \delta$ is~$h(\alpha) = \sum_{\delta \in \Delta} \alpha_\delta$.
The \defn{absolute height} of~$\alpha$ is~$|h|(\alpha) = |h(\alpha)|$.

The root system~$\Phi$ is \defn{crystallographic} if~$\dotprod{\alpha^\vee}{\beta} \in \Z$ for any~$\alpha, \beta \in \Phi$.
Equivalently, the Coxeter group~$W$ stabilizes the lattice~$\Z\Phi$, and is called a \defn{Weyl group}.
In most of the paper, we restrict our attention to crystallographic root systems.
Remarks~\ref{rem:failSums}, \ref{rem:failFlag}, \ref{rem:failCaracterizationClosed}, \ref{rem:failLatticeClosed} and~\ref{rem:failLatticePosets} justify this restriction.

\begin{example}[Type~$A$]
Let~$(\b{e}_i)_{i \in [n+1]}$ be the standard basis of~$\R^{n+1}$.
The symmetric group~$\fS_{n+1}$ acts on~$\R^{n+1}$ by permutation of coordinates.
It is the Weyl group of type~$A_n$.
The roots are $\Phi_{A_n} = \set{\b{e}_i - \b{e}_j}{1 \le i \ne j \le n+1}$, the positive roots are~$\Phi_{A_n}^+ = \set{\b{e}_i - \b{e}_j}{1 \le i < j \le n+1}$ and the simple roots are~$\Delta_{A_n} = \set{\b{e}_i - \b{e}_{i+1}}{i \in [n]}$.
A subset of~$\Phi_{A_n}$ can thus be identified with a binary relation on~$[n]$ via the bijection~$(i,j) \in [n]^2 \; \longleftrightarrow \; \b{e}_i-\b{e}_j \in \Phi_A$.
Note that the height of~$\b{e}_i - \b{e}_j$ is~$j-i$.
\end{example}

\subsection{Sums of roots in crystallographic root systems}
\label{subsec:sums}

We now gather statements on sums of roots in crystallographic root systems that are needed throughout the paper and that we consider interesting for their own sake.
We start by a statement from~\cite{Bourbaki} providing sufficient conditions for the sum or difference of two roots to be again a root in a crystallographic root system~$\Phi$.

\begin{theorem}[{\cite[Chap.~6, 1.3, Thm.~1]{Bourbaki}}]
\label{thm:dpsum}
For any~$\alpha, \beta$ in a crystallographic root system~$\Phi$,
\begin{enumerate}[(i)]
\item if~$\dotprod{\alpha}{\beta} > 0$ then~$\alpha - \beta \in \Phi$ or~$\alpha = \beta$,
\item if~$\dotprod{\alpha}{\beta} < 0$ then~$\alpha + \beta \in \Phi$ or~$\alpha = -\beta$.
\end{enumerate}
\end{theorem}

We say that a (multi)set~$\rel[X] \subseteq \Phi$ 
\begin{itemize}
\item is \defn{summable} if its sum~$\Sigma{\rel[X]}$ is again a root of~$\Phi$,
\item has \defn{no vanishing subsum} if~$\Sigma{\rel[Y]} \ne 0$ for any~$\varnothing \ne \rel[Y] \subseteq \rel[X]$.
\end{itemize}
Proposition~\ref{prop:twoSubsumsRoots} and Theorems~\ref{thm:filtrationSummableSubsets} and~\ref{thm:manySummableSubsets} ensure that a summable set of roots with no vanishing subsum has many summable subsets.
We start with sums of three roots.

\begin{proposition}
\label{prop:twoSubsumsRoots}
Let~$\Phi$ be a crystallographic root system.
If~$\alpha, \beta, \gamma \in \Phi$ are such that~$\alpha + \beta + \gamma \in \Phi$ has no vanishing subsum, then at least two of the three subsums~$\alpha + \beta$, $\alpha + \gamma$ and~$\beta + \gamma$~are~in~$\Phi$.
\end{proposition}

\begin{proof}
Assume by means of contradiction that~$\alpha + \beta \notin \Phi$ and~$\alpha + \gamma \notin \Phi$.
Since~$\alpha + \beta + \gamma$ has no vanishing subsum, $\alpha \ne -\beta$ and $\alpha \ne -\gamma$.
By contraposition of Theorem~\ref{thm:dpsum}\,(ii), we obtain that~$\dotprod{\alpha}{\beta} \ge 0$ and~$\dotprod{\alpha}{\gamma} \ge 0$.
Therefore, $\dotprod{\alpha + \beta + \gamma}{\beta + \gamma} = \dotprod{\alpha}{\beta} + \dotprod{\alpha}{\gamma} + \dotprod{\beta + \gamma}{\beta + \gamma} > 0$,
since~$\beta + \gamma \ne 0$.
It follows that either~$\dotprod{\alpha + \beta + \gamma}{\beta} > 0$ or $\dotprod{\alpha + \beta + \gamma}{\gamma} > 0$.
Assume for instance~$\dotprod{\alpha + \beta + \gamma}{\beta} > 0$.
Theorem~\ref{thm:dpsum}\,(i) thus implies that either~$\alpha + \gamma \in \Phi$ or~$\alpha + \gamma = 0$ contradicting either of our assumptions on $\alpha + \gamma$.
\end{proof}

It is proved in~\cite[Chap.~6, 1.6, Prop.~19]{Bourbaki} that any summable subset~$\rel[X]$ of positive roots admits a filtration of summable subsets~$\rel[X]_1 \subsetneq \rel[X]_2 \subsetneq \dots \subsetneq \rel[X]_{|{\rel[X]}|-1} \subsetneq \rel[X]_{|{\rel[X]}|} = \rel[X]$.
We now use Proposition~\ref{prop:twoSubsumsRoots} to extend this property in two directions: first we consider subsets of all roots (positive and negative), second we show that we can additionally prescribe the initial set~$\rel[X]_1$ to be a chosen root of~$\Phi$.
This latter improvement will be crucial all throughout the paper.

\begin{theorem}
\label{thm:filtrationSummableSubsets}
Let~$\Phi$ be a crystallographic root system.
Any summable set~$\rel[X] \subseteq \Phi$ with no vanishing subsum admits a filtration of summable subsets $\{\alpha\} = \rel[X]_1 \subsetneq \rel[X]_2 \subsetneq \dots \subsetneq \rel[X]_{|{\rel[X]}|-1} \subsetneq \rel[X]_{|{\rel[X]}|} = \rel[X]$ for any~$\alpha \in \rel[X]$.
\end{theorem}

\begin{proof}
The proof works by induction on~$|{\rel[X]}|$.
It is clear for~$|{\rel[X]}| = 2$, so that we consider~$|{\rel[X]}| > 2$.
By induction, it suffices to find a summable subset~$\rel[X]_{|{\rel[X]}|-1}$ of size~$|{\rel[X]}|-1$ such that~$\alpha \in \rel[X]_{|{\rel[X]}|-1} \subset \rel[X]$.
Since~$\sum_{\beta \in {\rel[X]}} \dotprod{\beta}{\Sigma {\rel[X]}} = \dotprod{\Sigma {\rel[X]}}{\Sigma {\rel[X]}} > 0$, there exists~$\beta \in {\rel[X]}$ such that~$\dotprod{\beta}{\Sigma {\rel[X]}} > 0$.
Since~$\rel[X]$ has no vanishing subsum, $\beta \ne \Sigma{\rel[X]}$.
Theorem~\ref{thm:dpsum}\,(i) thus ensures that~$\rel[X] \ssm \{\beta\}$ is summable.
If~$\alpha \ne \beta$, then we set~$\rel[X]_{|{\rel[X]}|-1} \eqdef \rel[X] \ssm \{\beta\}$ and conclude by induction.
Otherwise, we proved that both~$\{\alpha\}$ and~$\rel[X] \ssm \{\alpha\}$ are summable.
Let~$\rel[Y]$ be inclusion maximal with~$\alpha \in \rel[Y] \subsetneq \rel[X]$ such that both~$\rel[Y]$ and~$\rel[X] \ssm \rel[Y]$ are summable.
Assume that~$|{\rel[X] \ssm \rel[Y]}| \ge 2$.
By induction hypothesis, there exists~$\rel[Z] \subset \rel[X] \ssm \rel[Y]$ summable with~$|{\rel[Z]}| = |{\rel[X] \ssm \rel[Y]}| - 1 \ge 1$.
Let~$\gamma$ be the root in~$(\rel[X] \ssm \rel[Y]) \ssm \rel[Z]$.
Since~$\gamma$, $\Sigma{\rel[Y]}$ and~$\Sigma{\rel[Z]}$ are roots and~$\gamma + \Sigma{\rel[Y]} + \Sigma{\rel[Z]} = \Sigma{\rel[X]} \in \Phi$, Proposition~\ref{prop:twoSubsumsRoots} affirms that either~$\{\gamma\} \cup \rel[Y]$ or~$\rel[Y] \cup \rel[Z]$ is summable, contradicting the maximality of~$\rel[Y]$.
We therefore obtained a summable subset~$\rel[Y]$ with~$\alpha \in \rel[Y] \subseteq \rel[X]$ with~$|{\rel[Y]}| = |{\rel[X]}| - 1$.
We set~$\rel[X]_{|{\rel[X]}|-1} \eqdef \rel[Y]$ and conclude by induction.
\end{proof}

Finally, we obtain the following generalization of Proposition~\ref{prop:twoSubsumsRoots}.

\begin{theorem}
\label{thm:manySummableSubsets}
Let~$\Phi$ be a crystallographic root system.
Any summable set~$\rel[X] \subseteq \Phi$ with no vanishing subsum admits at least~$p$ distinct summable subsets of size~$|{\rel[X]}|-p+1$, for any~$1 \le p \le |{\rel[X]}|$.
\end{theorem}

\begin{proof}
Note that it holds for~$p = 1$ and~$p = |X|$.
We now proceed by induction on~$|X|$ to prove the result for~${1 < p < |X|}$.
By Theorem~\ref{thm:filtrationSummableSubsets}, $\rel[X]$ admits a summable subset~$\rel[Y]$ of size~$|{\rel[X]}|-1$.
Since~$1 < p$, we can apply the induction hypothesis to find $p-1$ distinct summable subsets~$\rel[Z]_1, \dots, \rel[Z]_{p-1}$ of~$\rel[Y]$ of size~$|{\rel[Y]}|-p+2 = |{\rel[X]}|-p+1$.
Moreover, by Theorem~\ref{thm:filtrationSummableSubsets} there exists at least one summable subset~$\rel[Z]_p$ of~$\rel[X]$ of size~$|{\rel[X]}|-p+1$ containing the root~$\alpha$ in~$\rel[X] \ssm \rel[Y]$.
This subset~$\rel[Z]_p$ is distinct from all the subsets~$\rel[Z]_1, \dots, \rel[Z]_{p-1}$ of~$\rel[Y]$, since it contains~$\alpha$.
This concludes the proof.
\end{proof}

\begin{remark}
\label{rem:failSums}
All results presented in this section fail for non-crystallographic root systems.
For example, consider the Coxeter group of type~$H_3$ with Dynkin diagram \raisebox{-.1cm}{\includegraphics[scale=.8]{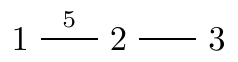}} and the positive roots~$\alpha \eqdef \alpha_1$, $\beta \eqdef \alpha_2$ and~$\gamma \eqdef s_1s_2s_3(\alpha_2) = \psi (\alpha_1 + \alpha_2 + \alpha_3)$, where~$\psi = -2\cos(4\pi/5)$.
Then
\begin{itemize}
\item $\dotprod{\alpha}{\beta} < 0$ while~$\alpha + \beta \notin \Phi$ and~$\alpha \ne -\beta$,
\item $\alpha + \beta + \gamma \in \Phi$ while~$\alpha + \beta \notin \Phi$ and $\beta + \gamma \notin \Phi$ (although $\alpha + \gamma \in \Phi$).
\end{itemize}
\end{remark}

\begin{remark}
\label{rem:failFlag}
For later purposes, we will need an even stronger counter-example to Theorem~\ref{thm:filtrationSummableSubsets} in non-crystallographic root systems.
Consider the Coxeter group of type~$H_2 = I_2(5)$ and the roots $\alpha \eqdef \alpha_1$, $\beta \eqdef \alpha_2$, $\gamma \eqdef \psi \alpha_1 + \psi \alpha_2$ and~$\delta \eqdef - \alpha_1 - \psi \alpha_2$, where~$\psi = -2\cos(4\pi/5)$.
It is not difficult to check that
\(
\Phi \cap \set{a\alpha + b\beta + c\gamma + d\delta}{a,b,c,d \in \N} = \{\alpha, \; \beta, \; \gamma, \; \delta, \; \alpha + \beta + \gamma + \delta\}.
\)
In particular, there is not even a single flag~$X_1 \subsetneq X_2 \subsetneq X_3 \subsetneq \{\alpha, \beta, \gamma, \delta\}$ of summable subsets of~$\{\alpha, \beta, \gamma, \delta\}$, even though $\{\alpha, \beta, \gamma, \delta\}$ is itself summable.
\end{remark}

\subsection{$\Phi$-posets}
\label{subsec:PhiPosets}

In Section~\ref{sec:weakOrderPosets}, we will consider certain specific families of collections of roots.
We start with the simple definition of symmetric and antisymmetric subsets of roots.

\begin{definition}
A subset~$\rel \subseteq \Phi$ is \defn{symmetric} if~$-\rel = \rel$ and \defn{antisymmetric} if~${\rel \cap -\rel = \varnothing}$.
We denote by~$\ISym(\Phi)$ (resp.~$\IAntisym(\Phi)$) the set of symmetric (resp.~antisymmetric) subsets of roots of~$\Phi$.
\end{definition}

We now want to define closed sets of roots.
The next statement is proved by A.~Pilkington~\cite[Sect.~2]{Pilkington} for subsets of positive roots.
We extend it to subsets of all roots using Theorem~\ref{thm:filtrationSummableSubsets}.

\begin{lemma}
\label{lem:closed}
In a crystallographic root system~$\Phi$, the following conditions are equivalent for~$\rel \subseteq \Phi$:
\begin{enumerate}[(i)]
\item $\alpha + \beta \in \rel$ for any~$\alpha, \beta \in \rel$ such that~$\alpha + \beta \in \Phi$,
\item $m\alpha + n\beta \in \rel$ for any~$\alpha, \beta \in \rel$ and~$m,n \in \N$ such that~$m\alpha + n\beta \in \Phi$,
\item $\alpha_1 + \dots + \alpha_p \in \rel$ for any~$\alpha_1, \dots, \alpha_p \in \rel$ such that~$\alpha_1 + \dots + \alpha_p \in \Phi$.
\end{enumerate}
\end{lemma}

\begin{proof}
The proof follows that of~\cite[Sect.~2]{Pilkington}.
The implications~(iii) $\Longrightarrow$ (ii) $\Longrightarrow$ (i) are clear.
Assume now that (i) holds and consider~$\alpha_1, \dots, \alpha_p \in \rel$ such that~$\alpha_1 + \dots + \alpha_p \in \Phi$.
By Theorem~\ref{thm:filtrationSummableSubsets}, there exists a flag~$\rel[X]_1 \subsetneq \rel[X]_2 \subsetneq \dots \subsetneq \rel[X]_p = \{\alpha_1, \dots, \alpha_p\}$ of summable subsets of~$\Phi$.
Applying inductively~(i), we obtain that~$\Sigma\rel[X]_i \in \rel$ for all~$i \in [p]$, and thus that~$\alpha_1 + \dots + \alpha_p \in \rel$.
\end{proof}

\begin{definition}
\label{def:closed}
In a crystallographic root system~$\Phi$, a subset~$\rel \subseteq \Phi$ is \defn{closed} if it satisfies the equivalent conditions of Lemma~\ref{lem:closed}. 
We denote by~$\IClo(\Phi)$ the set of closed subsets of roots of~$\Phi$.
\end{definition}

\begin{definition}
In a crystallographic root system~$\Phi$, the \defn{closure} of~$\rel \subseteq \Phi$ is the set~$\cl{\rel} \eqdef \N{\rel} \cap \Phi$.
\end{definition}

\begin{remark}
\label{rem:closureOperator}
The map~$\rel \mapsto \cl{\rel}$ is a closure operator on~$\Phi$, meaning that
\[
\cl{\varnothing} = \varnothing,
\qquad
\Phi \subseteq \cl{\Phi},
\qquad
\rel \subseteq \rel[S] \Longrightarrow \cl{\rel} \subseteq \cl{\rel[S]}
\qquad\text{and}\qquad
\cl{(\cl{\rel})} = \cl{\rel}
\]
for all~$\rel, \rel[S] \subseteq \Phi$.
Moreover~$\cl{\rel}$ is closed and~$\rel$ is closed if and only if~${\rel} = \cl{\rel}$.
\end{remark}

\begin{remark}
\label{rem:failCaracterizationClosed}
Lemma~\ref{lem:closed} fails for non-crystallographic root systems.
For example, consider the roots~$\alpha, \beta, \gamma, \delta$ of Remark~\ref{rem:failFlag}.
Then the set~$\rel[R] \eqdef \{\alpha, \beta, \gamma, \delta\}$ satisfies (i) but not (iii).
\end{remark}

\begin{remark}
\label{rem:manyNotionsClosed}
As studied in details by A.~Pilkington in~\cite{Pilkington}, even in crystallographic root systems, there are other possible notions of closed sets of roots.
Namely, one says that~$\rel \subseteq \Phi$ is
\begin{itemize}
\item \defn{$\N$-closed} if $m\alpha + n\beta \in \rel$ for any~$\alpha, \beta \in \rel$ and~$m,n \in \N$ such that~$m\alpha + n\beta \in \Phi$,
\item \defn{$\R$-closed} if $x\alpha + y\beta \in \rel$ for any~$\alpha, \beta \in \rel$ and~$x,y \in \R$ such that~$x\alpha + y\beta \in \Phi$,
\item \defn{convex} if $\rel = \Phi \cap C$ for a convex cone~$C$ in~$V$.
\end{itemize}
Note that convex implies $\R$-closed which implies $\N$-closed, but that the converse statements are wrong even for finite root systems~\cite[p.\,3192]{Pilkington}.
In this paper, we will only work with the notion of $\N$-closedness in crystallographic root systems, as it is discussed in~\cite{Bourbaki}.
Remarks~\ref{rem:failLatticeConvex} and~\ref{rem:failLatticeConvexAntisym} justify this restriction.
\end{remark}

\begin{example}[Type~$A$]
Identify subsets of roots with integer binary relations via the bijection~$(i,j) \in [n]^2 \; \longleftrightarrow \; \b{e}_i-\b{e}_j \in \Phi_A$.
A subset of roots is symmetric (resp.~antisymmetric, resp.~closed) if the corresponding integer binary relation is symmetric (resp.~antisymmetric, resp.~transitive).
(Note that here the three notions of closed sets of roots coincide in type~$A$.)
\end{example}

This example motivates the definition of the central object of this paper.

\begin{definition}
\label{def:posets}
In a crystallographic root system~$\Phi$, a \defn{$\Phi$-poset} is an antisymmetric and $\N$-closed subset of roots of~$\Phi$.
We denote by~$\IPos(\Phi)$ the set of all $\Phi$-posets.
\end{definition}

We speak of Weyl posets when we do not want to specify the root system.
We will introduce in Section~\ref{subsec:weakOrderPosets} a natural lattice structure on $\Phi$-posets.
We will see in Section~\ref{sec:examples} various subfamilies of $\Phi$-posets arising from classical Coxeter and Coxeter Catalan combinatorics.

To conclude this preliminary section on $\Phi$-posets, we gather simple observations on their subsums and their extensions.

\begin{lemma}
\label{lem:posetNoVanishingSubsum}
For any $\Phi$-poset~$\rel$ and any~$\alpha_1, \dots, \alpha_p \in \rel$, we have~$\alpha_1 + \dots + \alpha_p \ne 0$.
\end{lemma}

\begin{proof}
Assume that ~$\rel$ is a $\Phi$-poset and that there are~$\alpha_1, \dots, \alpha_p \in \rel$ such that~$\alpha_1 + \dots + \alpha_p = 0$.
Then~$\alpha_2 + \dots + \alpha_p = -\alpha_1$ is a root, so Lemma~\ref{lem:closed}\,(iii) ensures that~$\alpha_2 + \dots + \alpha_p \in \rel$ since~$\rel$ is closed.
We obtain that~$\alpha_1 \in {\rel}$ and~$-\alpha_1 \in {\rel}$, contradicting the antisymmetry of~$\rel$.
\end{proof}

Finally, we need $\Phi$-poset extensions.
The subsets of~$\Phi$ are naturally ordered by inclusion, and we consider the restriction of this inclusion order on $\Phi$-posets.
For~$\rel \in \IPos(\Phi)$, we call \defn{extensions}~of~$\rel$ the $\Phi$-posets~$\rel[S]$ containing~$\rel[R]$, and we let~$\extensions(\rel) \eqdef \set{\rel[S] \in \IPos(\Phi)}{\rel[R] \subseteq \rel[S]}$.
Note that~${\rel \subseteq \bigcap \extensions(\rel)}$ but the reverse inclusion does not always hold (consider for example~$\rel \eqdef \{\alpha_1+\alpha_2, \, \alpha_2\}$ in type~$B_2$).
For later purposes, we are interested in maximal $\Phi$-posets in the extension order.

\begin{proposition}
\label{prop:maxExtension}
For~$\rel \in \IPos(\Phi)$, we have~$\extensions(\rel) = \{\rel\}$ if and only if~$\{\alpha, -\alpha\} \cap {\rel} \ne \varnothing$ for all~$\alpha \in \Phi$.
\end{proposition}

\begin{proof}
Clearly if~$\{\alpha, -\alpha\} \cap {\rel} \ne \varnothing$ for all~$\alpha \in \Phi$, then adding any root to~$\rel$ breaks the antisymmetry, so that~$\extensions(\rel) = \{\rel\}$.
Reciprocally, assume that there exists~$\alpha \in \Phi$ such that~$\{\alpha, -\alpha\} \cap \Phi = \varnothing$.
Let~$\rel[S] \eqdef \cl{({\rel} \cup \{\alpha\})}$ and~$\rel[T] \eqdef \cl{({\rel} \cup \{-\alpha\})}$.
By definition, both~$\rel[S]$ and~$\rel[T]$ are closed, and we claim that at least one of them is antisymmetric, thus proving that~$\rel$ admits a non-trivial extension.
Assume by means of contradiction that neither~$\rel[S]$ nor~$\rel[T]$ are antisymmetric.
Let~$\beta \in {\rel[S]} \cap -{\rel[S]}$ and~$\gamma \in {\rel[T]} \cap -{\rel[T]}$.
By definition of the closure, we can write
\[
\beta = \sum_{\delta \in {\rel}} \lambda_\delta \delta + \lambda_\alpha \alpha = - \sum_{\delta \in {\rel}} \kappa_\delta \delta - \kappa_\alpha \alpha
\qquad\text{and}\qquad
\gamma = \sum_{\delta \in {\rel}} \mu_\delta \delta - \mu_\alpha \alpha = - \sum_{\delta \in {\rel}} \nu_\delta \delta + \nu_\alpha \alpha,
\]
where~$\lambda_\delta, \kappa_\delta, \mu_\delta, \nu_\delta$ are non-negative integer coefficients for all~$\delta \in {\rel} \cup \{\alpha\}$. Moreover, we have ${\lambda_\alpha + \kappa_\alpha \ne 0 \ne \mu_\alpha + \nu_\alpha}$ since~$\rel$ is antisymmetric and closed.
This implies that
\[
\sum_{\delta \in {\rel}} \big( (\lambda_\alpha + \kappa_\alpha)(\mu_\delta + \nu_\delta) + (\mu_\alpha + \nu_\alpha)(\lambda_\delta + \kappa_\delta) \big) \, \delta = 0.
\]
Lemma~\ref{lem:posetNoVanishingSubsum} thus ensures that~$(\lambda_\alpha + \kappa_\alpha)(\mu_\delta + \nu_\delta) + (\mu_\alpha + \nu_\alpha)(\lambda_\delta + \kappa_\delta) = 0$ which in turns implies that~$\lambda_\delta = \kappa_\delta = \mu_\delta = \nu_\delta = 0$ for all~$\delta \in {\rel}$, a contradiction.
\end{proof}

\section{Weak order on $\Phi$-posets}
\label{sec:weakOrderPosets}

\subsection{Weak order on all subsets}

Let~$\Phi$ be a finite root system (not necessarily crystallographic for the moment), with positive roots~$\Phi^+$ and negative roots~$\Phi^-$.
We denote by~$\IRel(\Phi)$ the set of all subsets of~$\Phi$.
For~$\rel \in \IRel(\Phi)$, we denote by~$\rel^+ \eqdef {\rel} \cap \Phi^+$ its positive part and~$\rel^- \eqdef {\rel} \cap \Phi^-$ its negative part.
The following order was considered in type~$A$ in~\cite{ChatelPilaudPons}.

\begin{definition}
\label{def:weakOrder}
The \defn{weak order} on~$\IRel(\Phi)$ is defined by~$\rel \wole \rel[S]$ $\iff$ ${\rel^+ \supseteq \rel[S]^+}$ and~${\rel^- \subseteq \rel[S]^-}$.
\end{definition}

\begin{remark}
The name for this order relation will be transparent in Section~\ref{sec:examples}.
Note that there is an arbitrary choice of orientation in Definition~\ref{def:weakOrder}.
The choice we have made here may seem unusual, as the apparent contradiction in Proposition~\ref{prop:weakOrderWOEP} suggests.
However, it is more coherent with the case of type~$A$ as treated in~\cite{ChatelPilaudPons} and it simplifies the presentation of Section~\ref{subsubsec:WOFP}.
\end{remark}

\begin{proposition}
\label{prop:all}
The weak order~$\wole$ on~$\IRel(\Phi)$ is a lattice with meet and join
\[
{\rel[R]} \meetR {\rel[S]} = ( \rel[R]^+ \cup \rel[S]^+ ) \sqcup ( \rel[R]^- \cap \rel[S]^- )
\qquad\text{and}\qquad
{\rel[R]} \joinR {\rel[S]} = ( \rel[R]^+ \cap \rel[S]^+ ) \sqcup ( \rel[R]^- \cup \rel[S]^- ).
\]
Furthermore, it is graded by $\rel[R] \mapsto |{\rel[R]^-}| - |{\rel[R]^+}|$ and its cover relations are given by
\[
\rel[R] \wole \rel[R] \ssm \{\alpha\} \text{ for } \alpha \in \rel[R]^+
\qquad\text{and}\qquad
\rel[R] \ssm \{\beta\} \wole \rel[R] \text{ for } \beta \in \rel[R]^-.
\]
\end{proposition}

\begin{proof}
It is the Cartesian product of two boolean lattices (the reverse inclusion poset on the positive roots and the inclusion poset on the negative roots).
\end{proof}

This section is devoted to show that the restriction of the weak order to certain families of subsets of roots (antisymmetric, closed and $\Phi$-posets) still defines a lattice structure when~$\Phi$ is crystallographic and to express its meet and join operations.
For example, the weak orders on $A_2$-, $B_2$- and $G_2$-posets are represented in Figures~\ref{fig:posetsAB} and~\ref{fig:posetsG}.

\subsection{Weak order on antisymmetric subsets}
\label{subsec:weakOrderAntisymmetric}

\enlargethispage{.4cm}
We start with the antisymmetry condition.

\begin{proposition}
\label{prop:antisymmetric}
The meet~$\meetR$ and the join~$\joinR$ both preserve antisymmetry. Thus, the set~$\IAntisym(\Phi)$ of antisymmetric subsets of~$\Phi$ induces a sublattice of the weak order on~$\IRel(\Phi)$.
\end{proposition}

\begin{proof}
Consider two antisymmetric subsets~$\rel[R], \rel[S] \in \IRel(\Phi)$ and let~$\alpha \in (\rel[R] \meetR \rel[S])^+ = \rel[R]^+ \cup \rel[S]^+$.
Assume for instance~$\alpha \in \rel[R]^+$.
Since~$\rel[R]$ is antisymmetric, $-\alpha \notin \rel[R]^-$, so that~$-\alpha \notin \rel[R]^- \cap \rel[S]^- = (\rel[R] \meetR \rel[S])^-$.
We conclude that~$\rel[R] \meetR \rel[S]$ is antisymmetric.
The proof for~$\rel[R] \joinR \rel[S]$ is similar.
\end{proof}

\begin{proposition}
\label{prop:coverRelationsAntisymmetric}
All cover relations in the weak order on~$\IAntisym(\Phi)$ are cover relations in the weak order on~$\IRel(\Phi)$. In particular, the weak order on~$\IAntisym(\Phi)$ is still graded by~$\rel \mapsto |{\rel^+}|-|{\rel^-}|$.
\end{proposition}

\begin{proof}
Consider a cover relation~$\rel[R] \wole \rel[S]$ in the weak order on~$\IAntisym(\Phi)$. We have~$\rel[R]^+ \supseteq \rel[S]^+$ and~$\rel[R]^- \subseteq \rel[S]^-$ where at least one of the inclusions is strict. Suppose first that~$\rel[R]^+ \ne \rel[S]^+$. Let~${\alpha \in \rel[R]^+ \ssm \rel[S]^+}$ and~$\rel[T] \eqdef \rel[R] \ssm \{\alpha\}$. Note that~$\rel[T] \in \IAntisym(\Phi)$ and~$\rel[R] \prec \rel[T] \wole \rel[S]$. Since~$\rel[S]$ covers~$\rel$, we get~${\rel[S] = \rel[T] = \rel[R] \ssm \{\alpha\}}$. Similarly if $\rel[S]^- \ne \rel[R]^-$ let $\alpha \in \rel[S]^- \ssm \rel[R]^-$ and $\rel[T] \eqdef \rel[S]^- \ssm \{\alpha\}$. Then~$\rel[T] \in \IAntisym(\Phi)$ and $\rel[R] \wole \rel[T] \prec \rel[S]$ implies that~$\rel[T] = \rel[R] = \rel[S] \ssm \{\alpha\}$. 
In both cases, $\rel[R] \wole \rel[S]$ is a cover relation of the weak order on~$\IRel(\Phi)$.
\end{proof}

\begin{corollary}
\label{coro:coverRelationsAntisymmetric}
In the weak order on~$\IAntisym(\Phi)$, the antisymmetric subsets that cover a given antisymmetric subset~$\rel \in \IAntisym(\Phi)$ are precisely the relations
\begin{itemize}
\item $\rel \ssm \{\alpha\}$ for any~$\alpha \in \rel^+$,
\item $\rel \cup \{\beta\}$ for any~$\beta \in \Phi^- \ssm \rel^-$ such that~$-\beta \notin \rel^+$.
\end{itemize}
\end{corollary}

\subsection{Weak order on closed subsets}
\label{subsec:weakOrderClosed}

We want to prove that the weak order on closed subsets of~$\Phi$ is also a lattice.
Contrarily to Propositions~\ref{prop:all} and \ref{prop:antisymmetric}, we now need to assume that the root system~$\Phi$ is crystallographic (see Remarks~\ref{rem:failCaracterizationClosed}, \ref{rem:failLatticeClosed} and~\ref{rem:failLatticePosets}).
Unfortunately, as $\IClo(\Phi)$ is stable by intersection but not by union, it is not preserved by the meet~$\meetR$ and the join~$\joinR$, so that it does not induce a sublattice of the weak order on~$\IRel(\Phi)$.
Proving that it is still a lattice requires more work.
Following~\cite{ChatelPilaudPons}, we start with a weaker notion of closedness.
We say that a subset $\rel = \rel^+ \sqcup \rel^-$ is \defn{semiclosed} if both~$\rel^+$ and~$\rel^-$ are closed.
We denote by~$\ISemiClo(\Phi)$ the set of semiclosed subsets of~$\Phi$.
Note that $\IClo(\Phi) \subseteq \ISemiClo(\Phi)$ but that the reverse inclusion does not hold~in~general.

\begin{proposition}
\label{prop:semiclosed}
The weak order~$\wole$ on~$\ISemiClo(\Phi)$ is a lattice with meet and join
\[
{\rel[R]} \meetSC {\rel[S]} = \cl{( \rel^+ \cup \rel[S]^+ )} \sqcup ( \rel^- \cap \rel[S]^- )
\quad\text{and}\quad
{\rel[R]} \joinSC {\rel[S]} = ( \rel^+ \cap \rel[S]^+ ) \sqcup \cl{( \rel^- \cup \rel[S]^- )}.
\]
\end{proposition}

\begin{proof}
Observe first that ${\rel[R]} \meetSC {\rel[S]}$ is indeed semiclosed ($\cl{\rel[T]}$ is always closed and $\IClo(\Phi)$ is stable by intersection).
Moreover, ${\rel[R]} \meetSC {\rel[S]} \wole \rel[R]$ and~${\rel[R]} \meetSC {\rel[S]} \wole \rel[S]$.
Assume now that~$\rel[T] \subseteq \Phi$ is semiclosed such that~$\rel[T] \wole \rel[R]$ and~$\rel[T] \wole \rel[S]$.
Then~$\rel[T]^+ \supseteq \rel[R]^+ \cup \rel[S]^+$ and $\rel[T]^- \subseteq \rel[R]^- \cap \rel[S]^-$.
Moreover, since~$\rel[T]^+$ is closed, we get that~$\rel[T]^+ \supseteq \cl{( \rel[R]^+ \cup \rel[S]^+ )}$ so that~$\rel[T] \wole {\rel[R]} \meetSC {\rel[S]}$.
We conclude that~${\rel[R]} \meetSC {\rel[S]}$ is indeed the meet of~$\rel[R]$ and~$\rel[S]$.
The proof is similar for the join.
\end{proof}

\begin{proposition}
\label{prop:coverRelationsSemiClosed}
All cover relations in the weak order on~$\ISemiClo(\Phi)$ are cover relations in the weak order on~$\IRel(\Phi)$. In particular, the weak order on~$\ISemiClo(\Phi)$ is still graded by $\rel[R] \mapsto |{\rel[R]^-}| - |{\rel[R]^+}|$.
\end{proposition}

\begin{proof}
Consider a cover relation $\rel[R] \wole \rel[S]$ in the weak order on~$\ISemiClo(\Phi)$.
We have $\rel[R]^+ \supseteq \rel[S]^+$ and $\rel[R]^- \subseteq \rel[S]^-$ where at least one of the inclusions is strict.
We distinguish two cases.

Suppose first that $\rel[R]^+ \ne \rel[S]^+$, and consider~$\alpha \in \rel[R]^+ \ssm \rel[S]^+$ of minimal height in $\rel[R]^+ \ssm \rel[S]^+$.
Observe that $\alpha$ cannot be decomposed in $\rel[R]^+$: if~$\alpha = \gamma + \delta$ with $\gamma, \delta \in \rel[R]^+$, then $h(\gamma), h(\delta) < h(\alpha)$, so $\gamma, \delta \in \rel[S]^+$ by minimality of~$h(\alpha)$, which contradicts the closedness of~$\rel[S]^+$.
Consider now~$\rel[T] \eqdef \rel[R] \ssm \{\alpha\}$.
Let~$\gamma, \delta \in \rel[T]^+$ with~$\gamma + \delta \in \Phi$.
Then~$\gamma, \delta \in \rel[R]^+$ so that~$\gamma + \delta \in \rel[R]^+$ since~$\rel[R]^+$ is closed.
Since~$\gamma + \delta \ne \alpha$, this implies that~$\gamma + \delta \in \rel[T]^+$.
This shows that~$\rel[T]^+$ is closed.
Since~$\rel[T]^- = \rel[R]^-$ is also closed, we obtain that~$\rel[T]$ is semiclosed.
Since~$\rel[R] \neq \rel[T]$ and $\rel[R] \wole \rel[T] \wole \rel[S]$, this proves that $\rel[T] = \rel[S] = \rel[R] \ssm \{\alpha\}$.

Assume now that~$\rel[R]^- \ne \rel[S]^-$, and let~$\beta \in \rel[S]^- \ssm \rel[R]^-$ of minimal height (or equivalently maximal absolute height).
Consider~$\rel[T] \eqdef \rel[R] \cup \{\beta\}$.
Let~$\gamma, \delta \in \rel[T]^-$ with~$\gamma + \delta \in \Phi$.
If~$\gamma, \delta \in \rel[R]^-$, then~$\gamma + \delta \in \rel[R]^-$ since $\rel[R]^-$ is closed.
Assume now that~$\delta = \beta$.
Then~$\gamma, \beta \in \rel[S]^-$ and~$\rel[S]^-$ is closed, we have~$\gamma + \beta \in \rel[S]^-$ and $h(\gamma + \beta) < h(\beta)$, which ensures that~$\gamma + \beta \in \rel[R]^-$ by minimality of~$h(\beta)$.
This shows that~$\rel[T]^-$ is closed.
Since~$\rel[T]^+ = \rel[R]^+$ is also closed, we obtain that~$\rel[T]$ is semiclosed.
Since~$\rel[R] \neq \rel[T]$ and $\rel[R] \wole \rel[T] \wole \rel[S]$, this proves that $\rel[T] = \rel[S] = \rel[R] \cup \{\beta\}$.
\end{proof}

\begin{corollary}
\label{coro:coverRelationsSemiClosed}
In the weak order on~$\ISemiClo(\Phi)$, the semiclosed subsets of~$\Phi$ that cover a given semiclosed subset~$\rel \in \ISemiClo(\Phi)$ are precisely the relations:
\begin{itemize}
\item $\rel \ssm \{\alpha\}$ for any $\alpha \in \rel^+$ such that there is no~$\gamma, \delta \in \rel[R]^+$ with~$\alpha = \gamma + \delta$,
\item $\rel \cup \{\beta\}$ for any~$\beta \in \Phi^- \ssm \rel^-$ such that $\beta + \gamma \in \Phi \implies \beta + \gamma \in \rel$ for all $\gamma \in \rel^-$.
\end{itemize}
\end{corollary}

We now come back to closed subsets of~$\Phi$ introduced in Definition~\ref{def:posets}.
Unfortunately, $\IClo(\Phi)$ still does not induce a sublattice of~$\ISemiClo(\Phi)$.
We thus need a transformation similar to the closure~$\rel \mapsto \cl{\rel}$ to transform a semiclosed subset of~$\Phi$ into a closed one.
For~$\rel \in \IRel(\Phi)$, we define the \defn{negative closure deletion}~$\ncd{\rel}$ and the \defn{positive closure deletion}~$\pcd{\rel}$~by
\begin{align*}
\ncd{\rel} & \eqdef {\rel} \ssm \bigset{\alpha \in \rel^-}{\exists \, \rel[X] \subseteq \rel^+ \text{such that } \alpha + \Sigma \rel[X] \in \Phi \ssm \rel}, \\
\pcd{\rel} & \eqdef {\rel} \ssm \bigset{\alpha \in \rel^+}{\exists \, \rel[X] \subseteq \rel^- \text{such that } \alpha + \Sigma \rel[X] \in \Phi \ssm \rel}.
\end{align*}
As in Section~\ref{subsec:sums}, the notation $\Sigma{\rel[X]}$ in these formulas denotes the sum of all roots in~$\rel[X]$.

\begin{remark}
\label{rem:noVanishingSubsum}
In the case that~$\rel$ is semiclosed, we can assume that the set~$\rel[X]$ in the definitions of~$\ncd{\rel}$ and~$\pcd{\rel}$ is such that the~$\alpha + \Sigma{\rel[X]}$ has no vanishing subsum.
Observe first that no vanishing subsum can contain~$\alpha$.
Indeed, if~$\rel[Y] \subseteq \rel[X]$ is such that~$\alpha + \Sigma{\rel[Y]} = 0$, then $\rel[X] \ssm \rel[Y] \subseteq \rel^-$ and~$\rel^-$ closed implies that~$\alpha + \Sigma{\rel[X]} = \Sigma(\rel[X] \ssm \rel[Y]) \in \rel$.
Now if~$\rel[Y] \subseteq \rel[X]$ is such that~$\Sigma{\rel[Y]} = 0$, then $\alpha + \Sigma(\rel[X] \ssm \rel[Y]) = \alpha + \Sigma{\rel[X]} \notin \rel$, so that we can replace~$\rel[X]$ by~$\rel[X] \ssm \rel[Y]$.
\end{remark}

\begin{lemma}
\label{lem:orderNcdPcd}
For any~$\rel \in \IRel(\Phi)$, we have $\ncd{\rel} \wole \rel \wole \pcd{\rel}$.
\end{lemma}

\begin{proof}
Since~$\ncd{\rel}$ (resp.~$\pcd{\rel}$) is obtained from~$\rel$ by deleting negative (resp.~positive) roots, we have~$(\ncd{\rel})^+ = \rel^+ \supseteq (\pcd{\rel})^+$ and~$(\ncd{\rel})^- \subseteq \rel^- = (\pcd{\rel})^-$, so that~$\ncd{\rel} \wole \rel \wole \pcd{\rel}$.
\end{proof}

\begin{lemma}
\label{lem:semiclosedImpliesNcdPcdClosed}
If~$\Phi$ is crystallographic and~${\rel} \subseteq \Phi$ is semiclosed, then both~$\ncd{\rel}$ and~$\pcd{\rel}$ are closed.
\end{lemma}

\begin{proof}
Assume by means of contradiction that~$\rel$ is semiclosed and~$\ncd{\rel}$ is not closed.
Then there are roots~$\alpha, \beta \in \ncd{\rel}$ such that~$\alpha + \beta \in \Phi \ssm \ncd{\rel}$.
Consider two such roots such that~$\alpha + \beta$ has minimal absolute height.
We distinguish four cases:
\begin{itemize}
\item If~$\alpha, \beta \in \Phi^+$, then~$\alpha, \beta \in (\ncd{\rel})^+ = \rel^+$, which is closed, so that~${\alpha + \beta \in \rel^+ = (\ncd{\rel})^+}$. Contradiction.
\item If~$\alpha \in \Phi^-$ and~$\beta \in \Phi^+$, we distinguish again two cases:
	\begin{itemize}
	\item If~$\alpha + \beta \notin \rel$, then the set~$\{\beta\}$ ensures~$\alpha \notin \ncd{\rel}$. Contradiction.
	\item If~$\alpha + \beta \in \rel$, then since~$\alpha + \beta \in \rel \ssm \ncd{\rel}$, there exists~$\rel[X] \subseteq \rel^+$ such that ${\alpha + \beta + \Sigma{\rel[X]} \in \Phi \ssm \rel}$. Since~$\beta \in \rel^+$, the set~$\{\beta\} \cup \rel[X]$ ensures~$\alpha \notin \ncd{\rel}$. Contradiction.
	\end{itemize}
\item If~$\alpha \in \Phi^+$ and~$\beta \in \Phi^-$, the argument is symmetric.
\item If~$\alpha, \beta \in \Phi^-$, then~$\alpha + \beta \in \rel^-$ since~$\rel^-$ is closed. Since~$\alpha + \beta \in \rel \ssm \ncd{\rel}$, there exists~${\rel[X] \subseteq \rel^+}$ such that~$(\alpha + \beta) + \Sigma{\rel[X]} \in \Phi \ssm \rel$. By Remark~\ref{rem:noVanishingSubsum}, we can assume that~${(\alpha + \beta) + \Sigma{\rel[X]}}$ has no vanishing subsum. By Theorem~\ref{thm:filtrationSummableSubsets}, there exists~$\gamma \in \rel[X]$ such that~$\alpha + \beta + \gamma \in \Phi$. By Proposition~\ref{prop:twoSubsumsRoots}, we can assume without loss of generality that~$\beta + \gamma \in \Phi$. We now distinguish four cases:
	\begin{itemize}
	\item If~$\beta + \gamma \notin \rel$, then the set~$\{\gamma\}$ ensures~$\beta \notin \ncd{\rel}$. Contradiction.
	\item If~$\beta + \gamma \in \rel^+$, then~$\rel[T] = \{\beta + \gamma\} \cup (\rel[X] \ssm \{\gamma\}) \subseteq \rel^+$ and~$\alpha + \Sigma{\rel[T]} = \alpha + \beta + \Sigma{\rel[X]} \in \Phi \ssm \rel$ so that~$\alpha \notin \ncd{\rel}$. Contradiction.
	\item If~$\beta + \gamma \in \rel^- \ssm \ncd{\rel}$, then there exists~$\rel[T] \subseteq \rel^+$ such that~$\beta + \gamma + \Sigma{\rel[T]} \in \Phi \ssm \rel$. Since~$\gamma \in \rel^+$, the set~$\{\gamma\} \cup \rel[T]$ ensures that~$\beta \notin \ncd{\rel}$. Contradiction.
	\item If~$\beta + \gamma \in (\ncd{\rel})^-$, then we have~$\alpha \in \ncd{\rel}$ and~$\beta + \gamma \in \ncd{\rel}$ with~$\alpha + \beta + \gamma \in \Phi$. Moreover, $h(\alpha + \beta + \gamma) < h(\alpha + \beta)$ since~$\alpha + \beta \in \Phi^-$ while~$\gamma \in \Phi^+$ and~$\beta + \gamma \in \Phi^-$. By minimality in the choice of~$\alpha + \beta$, we obtain that~$\alpha + \beta + \gamma \in \ncd{\rel}$. Observe now that~$\rel[X] \ssm \{\gamma\} \subseteq \rel^+$ and~$\alpha + \beta + \gamma + \Sigma(\rel[X] \ssm \{\gamma\}) = \alpha + \beta + \Sigma{\rel[X]} \in \Phi \ssm \rel$. Therefore:
		\begin{itemize}
		\item If~$\alpha + \beta + \gamma$ is negative, the set~$\rel[X] \ssm \{\gamma\}$ ensures~$\alpha + \beta + \gamma \notin \ncd{\rel}$. Contradiction.
		\item If~$\alpha + \beta + \gamma$ is positive, then~$\rel^+$ is not closed. Contradiction.
		\end{itemize}
	\end{itemize}
\end{itemize}
In all cases, we have reached a contradiction.
We conclude that if~$\rel$ is semiclosed, then~$\ncd{\rel}$ is closed.
The proof is symmetric for~$\pcd{\rel}$.
\end{proof}

\begin{proposition}
\label{prop:closed}
When~$\Phi$ is crystallographic, the weak order on~$\IClo(\Phi)$ is a lattice with meet~and~join
\[
{\rel[R]} \meetC {\rel[S]} = \ncd{\big( \cl{(\rel[R]^+ \cup \rel[S]^+)} \sqcup (\rel[R]^- \cap \rel[S]^-) \big)}
\quad\text{and}\quad
{\rel[R]} \joinC {\rel[S]} = \pcd{\big( (\rel[R]^+ \cap \rel[S]^+) \sqcup \cl{(\rel[R]^- \cup \rel[S]^-)} \big)}.
\]
\end{proposition}

\begin{proof}
First, the weak order~$\wole$ on~$\IClo(\Phi)$ is a subposet of the weak order~$\wole$ on~$\IRel(\Phi)$, and it is bounded below by~$\Phi^+$ and above by~$\Phi^-$.
We therefore just need to show that there is a meet and a join and that they are given by the above formulas.

Let~$\rel[R], \rel[S] \in \IClo(\Phi)$ and~$\rel[M] = \rel[R] \meetSC \rel[S]$ so that~$\ncd{\rel[M]} = {\rel[R]} \meetC {\rel[S]}$.
Observe that we have~$\ncd{\rel[M]} \wole \rel[M] \wole \rel[R]$ and~$\ncd{\rel[M]} \wole \rel[M] \wole \rel[S]$ by Lemma~\ref{lem:orderNcdPcd}.
Moreover, since~$\rel[M]$ is semiclosed, $\ncd{\rel[M]}$ is closed by Lemma~\ref{lem:semiclosedImpliesNcdPcdClosed}.
Therefore, $\ncd{\rel[M]}$ is closed and below both~$\rel[R]$ and~$\rel[S]$.

Consider now~$\rel[T] \in \IClo(\Phi)$ such that~$\rel[T] \wole \rel[R]$ and~$\rel[T] \wole \rel[S]$.
Since~$\rel[T] \in \ISemiClo(\Phi)$ and~${\rel[M] = \rel[R] \meetSC \rel[S]}$, we have~$\rel[T] \wole \rel[M]$.
Therefore, $\rel[T]^+ \supseteq \rel[M]^+ = (\ncd{\rel[M]})^+$ and~$\rel[T]^- \subseteq \rel[M]^-$.
Assume by means of contradiction that~$\rel[T] \not\wole \ncd{\rel[M]}$.
Then we have~$\rel[T]^- \not\subseteq (\ncd{\rel[M]})^-$.
Consider~$\alpha \in \rel[T]^- \ssm (\ncd{\rel[M]})^-$ of minimal absolute height.
By definition of~$\ncd{\rel[M]}$, there exists~$\rel[X] \subseteq \rel[M]^+$ such that~${\alpha + \Sigma{\rel[X]} \in \Phi \ssm \rel[M]}$.
Since~${\rel[M]^+ = (\rel[R] \meetSC \rel[S])^+ = \cl{(\rel[R]^+ \cup \rel[S]^+)}}$, we can assume without loss of generality (up to developing each root of~$\rel[X]$) that~${\rel[X] \subseteq (\rel[R]^+ \cup \rel[S]^+)}$.
By Remark~\ref{rem:noVanishingSubsum}, we can moreover assume that~$\alpha + \Sigma{\rel[X]}$ has no vanishing subsum.
By Theorem~\ref{thm:filtrationSummableSubsets}, there exists~$\beta \in \rel[X]$ such that~$\alpha + \beta \in \Phi$.

Since~$\beta \in \rel[X] \subseteq (\rel[R]^+ \cup \rel[S]^+)$, we can assume that~$\beta \in \rel[R]^+$.
Since~${\alpha \in \rel[T]^- \subseteq \rel[R]^-}$, $\beta \in \rel[R]^+ \subseteq \rel[T]^+$ and both~$\rel[R]$ and~$\rel[T]$ are closed, we obtain that~$\alpha + \beta \in {\rel[R]} \cap {\rel[T]}$.
We now distinguish two cases:
\begin{itemize}
\item If~$\alpha + \beta$ is positive, then~$\alpha + \beta \in \rel[R]^+ \subseteq \rel[M]^+$. Since~$\rel[X] \ssm \{\beta\} \subseteq \rel[M]^+$ and~$\rel[M]^+$ is closed, we obtain that~$\alpha + \Sigma{\rel[X]} = (\alpha + \beta) + \Sigma(\rel[X] \ssm \{\beta\}) \in \rel[M]^+$. Contradicion.
\item If~$\alpha + \beta$ is negative, we have~$\alpha + \beta \in \rel[T]^-$. Moreover, $\alpha + \beta$ has smaller absolute height than~$\alpha$ since~$\alpha \in \Phi^-$, $\beta \in \Phi^+$ and~$\alpha + \beta \in \Phi^-$. By minimality in the choice of~$\alpha$, we obtain that~$\alpha + \beta \in \ncd{\rel[M]}$. Since~$\rel[X] \ssm \{\beta\} \subseteq \rel[M]^+$ this implies that~$\alpha + \Sigma{\rel[X]} = (\alpha + \beta) + \Sigma(\rel[X] \ssm \{\beta\}) \in \rel[M]$. Contradiction.
\end{itemize}
Since we reached a contradiction in both cases, we obtain that~$\rel[T] \wole \ncd{\rel[M]}$.
Hence, $\ncd{\rel[M]}$ is indeed the meet of~$\rel[R]$ and~$\rel[S]$ for the weak order on~$\IClo(\Phi)$.
The proof is similar for the join.
\end{proof}

\begin{remark}
In contrast to Propositions~\ref{prop:coverRelationsAntisymmetric} and~\ref{prop:coverRelationsSemiClosed} and Corollaries~\ref{coro:coverRelationsAntisymmetric} and~\ref{coro:coverRelationsSemiClosed}, the cover relations in the weak order on~$\IClo(\Phi)$ are more intricate and the weak order on~$\IClo(\Phi)$ is not graded in general.
\end{remark}

\begin{remark}
\label{rem:failLatticeClosed}
All results presented in this section fail for non-crystallographic root systems.
In view of Remark~\ref{rem:failCaracterizationClosed}, it might a priori depend of the notion of $\N$-closed subsets considered.
However, the following example works for either of the notions (i), (ii) and (iii) of Lemma~\ref{lem:closed}. \\[-.1cm] \indent
Consider the Coxeter group of type~$H_3$ with Dynkin diagram \raisebox{-.1cm}{\includegraphics[scale=.8]{DynkinH3}}.
Consider the roots~$\alpha \eqdef \alpha_1 \in \Phi^+$, $\beta \eqdef - \alpha_1 - \psi \alpha_2 \in \Phi^-$ and~$\gamma \eqdef - \psi \alpha_1 - \alpha_2 - \alpha_3 \in \Phi^-$, where~$\psi = -2\cos(4\pi/5)$.
Note that~$\beta + \gamma \in \Phi^-$ and $\alpha + \beta + \gamma \in \Phi^-$, while~$\alpha + \beta \notin \Phi$ and~$\alpha + \gamma \notin \Phi$.
Consider the sets~$\rel[R] \eqdef \{\alpha, \; \beta, \; \gamma, \; \beta + \gamma, \; \alpha + \beta + \gamma\}$, $\rel[S] \eqdef \{\beta, \; \gamma, \; \beta + \gamma\}$, $\rel[U] \eqdef \{\alpha, \, \beta\}$ and~$\rel[V] \eqdef \{\alpha, \, \gamma\}$.
Note that~$\rel[R], \rel[S], \rel[U]$ and~$\rel[V]$ are closed, and that both~$\rel[U]$ and~$\rel[V]$ are weak order smaller than both~$\rel[R]$ and~$\rel[S]$.
Moreover, we claim that there is no closed subset~$\rel[T]$ which is weak order larger than both~$\rel[U]$ and~$\rel[V]$ and weak order smaller than both~$\rel[R]$ and~$\rel[S]$.
Indeed, such a set~$\rel[T]$ should contain~$\alpha, \beta, \gamma$ and thus~$\beta + \gamma$ and $\alpha + \beta + \gamma$ by closedness, which would contradict~$\rel[T] \wole \rel[S]$.
This implies that~$\rel[R]$ and~$\rel[S]$ have no meet and that~$\rel[U]$ and~$\rel[V]$ have no join in the weak order on closed subsets of~$\Phi$, thus contradicting the result of Proposition~\ref{prop:closed} in the non-crystallographic type~$H_3$.
In fact, even Lemma~\ref{lem:semiclosedImpliesNcdPcdClosed} fails in type~$H_3$ since~$\ncd{\{\alpha, \beta, \gamma, \beta + \gamma\}} = \{\alpha, \beta, \gamma\}$ is not closed.
\end{remark}

\begin{remark}
\label{rem:failLatticeConvex}
As mentioned in Remark~\ref{rem:manyNotionsClosed}, even for crystallographic root systems, there are different possible notions of closed subsets (which all coincide in type~$A$).
Unfortunately, it turns out that Proposition~\ref{prop:closed} fails for the other notions of closed sets.
The smallest counter-example is in type~$B_3$.
Consider the sets of roots
$\rel[R] \eqdef \{-\alpha_1, \; -\alpha_1 - \alpha_2, \; -\alpha_1 - \alpha_2 - \alpha_3, \; -\alpha_1 - 2\alpha_2 - 2\alpha_3, \; \alpha_3\}$,
$\rel[S] \eqdef \{-\alpha_1, \; -\alpha_1 - \alpha_2 - \alpha_3, \; -\alpha_1 - 2\alpha_2 - 2\alpha_3\}$,
$\rel[U] \eqdef \{-\alpha_1, \; \alpha_3\}$
and $\rel[V] \eqdef \{-\alpha_1 - 2\alpha_2 - 2\alpha_3, \; \alpha_3\}$.
Note that~$\rel[R], \rel[S], \rel[U]$ and~$\rel[V]$ are convex, and that both~$\rel[U]$ and~$\rel[V]$ are weak order smaller than both~$\rel[R]$ and~$\rel[S]$.
We have~$\rel[U] \joinC \rel[V] = \rel[R] \meetC \rel[S] = \{-\alpha_1, \; -\alpha_1 - 2\alpha_2 - 2\alpha_3, \; \alpha_3\}$ but this set is not convex.
In fact, we claim that there is no convex subset~$\rel[T]$ which is weak order larger than both~$\rel[U]$ and~$\rel[V]$ and weak order smaller than both~$\rel[R]$ and~$\rel[S]$.
Indeed, such a set~$\rel[T]$ should contain~$\{-\alpha_1, \; -\alpha_1 - 2\alpha_2 - 2\alpha_3, \; \alpha_3\}$ and thus also~$-\alpha_1 - \alpha_2 = (-\alpha_1)/2 + (-\alpha_1 - 2\alpha_2 - 2\alpha_3)/2 + \alpha_3$, contradicting~$\rel[T] \wole \rel[S]$.
This implies that~$\rel[R]$ and~$\rel[S]$ have no meet and that~$\rel[U]$ and~$\rel[V]$ have no join in the weak order on convex subsets of~$\Phi$.
\end{remark}

\subsection{Weak order on $\Phi$-posets}
\label{subsec:weakOrderPosets}
Recall from Definition~\ref{def:posets} that~$\IPos(\Phi)$ denotes the set of $\Phi$-posets, \ie of antisymmetric closed subsets of~$\Phi$.
We finally show that the restriction of the weak order to the $\Phi$-posets still defines a lattice structure.
The weak orders on $A_2$-, $B_2$- and $G_2$-posets are represented in Figures~\ref{fig:posetsAB} and~\ref{fig:posetsG}.
\hvFloat[floatPos=p, capWidth=h, capPos=r, capAngle=90, objectAngle=90, capVPos=c, objectPos=c]{figure}
{\begin{minipage}{23cm}\vspace*{-1.5cm} \includegraphics[scale=.9]{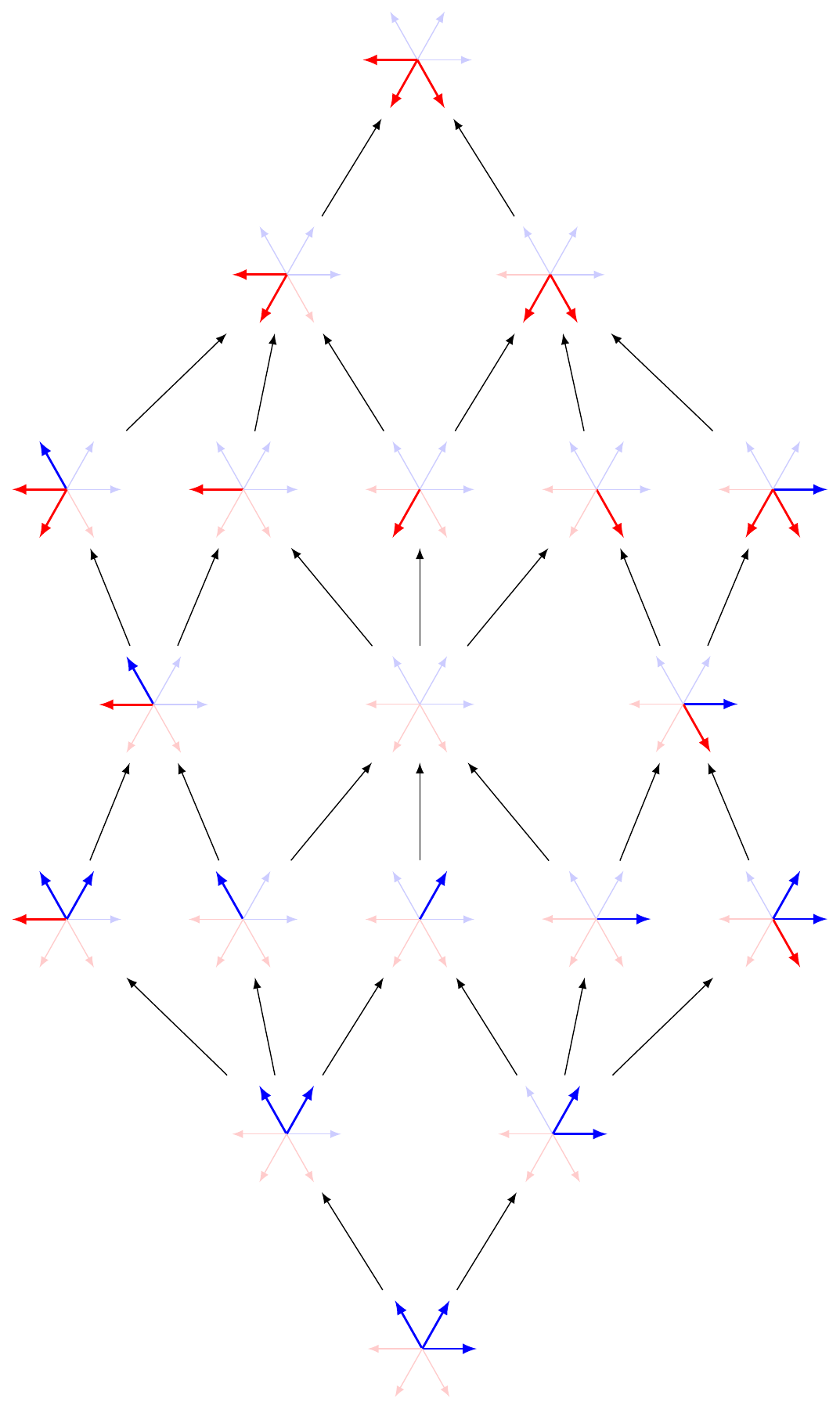} \hspace{2cm} \includegraphics[scale=.65]{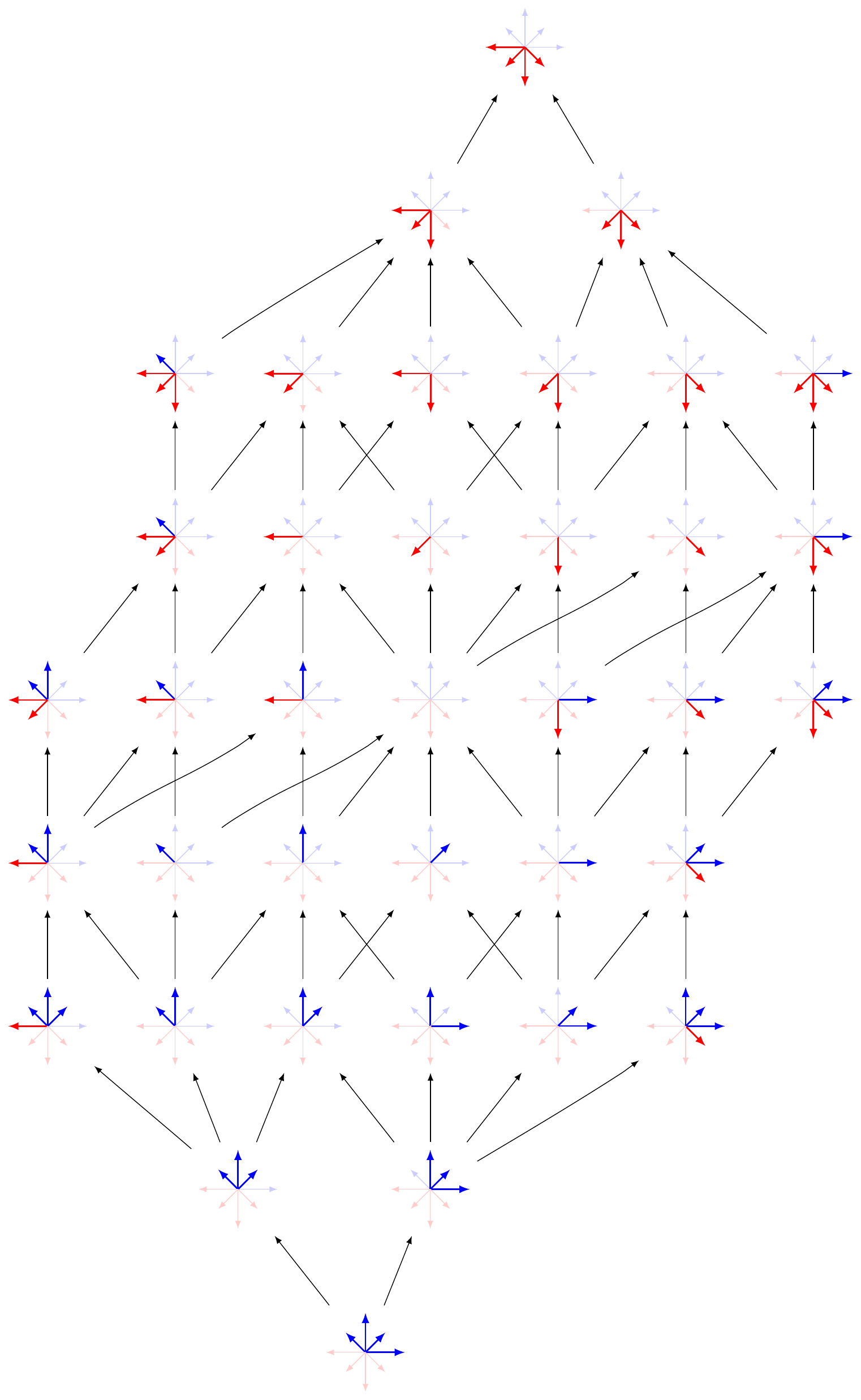} \vspace*{-.3cm}\end{minipage}}
{The weak order on $A_2$-posets (left) and on $B_2$-posets (right).}
{fig:posetsAB}
\hvFloat[floatPos=p, capWidth=h, capPos=r, capAngle=90, objectAngle=90, capVPos=c, objectPos=c]{figure}
{\begin{minipage}{23cm}\vspace*{-2.3cm}\includegraphics[scale=.4]{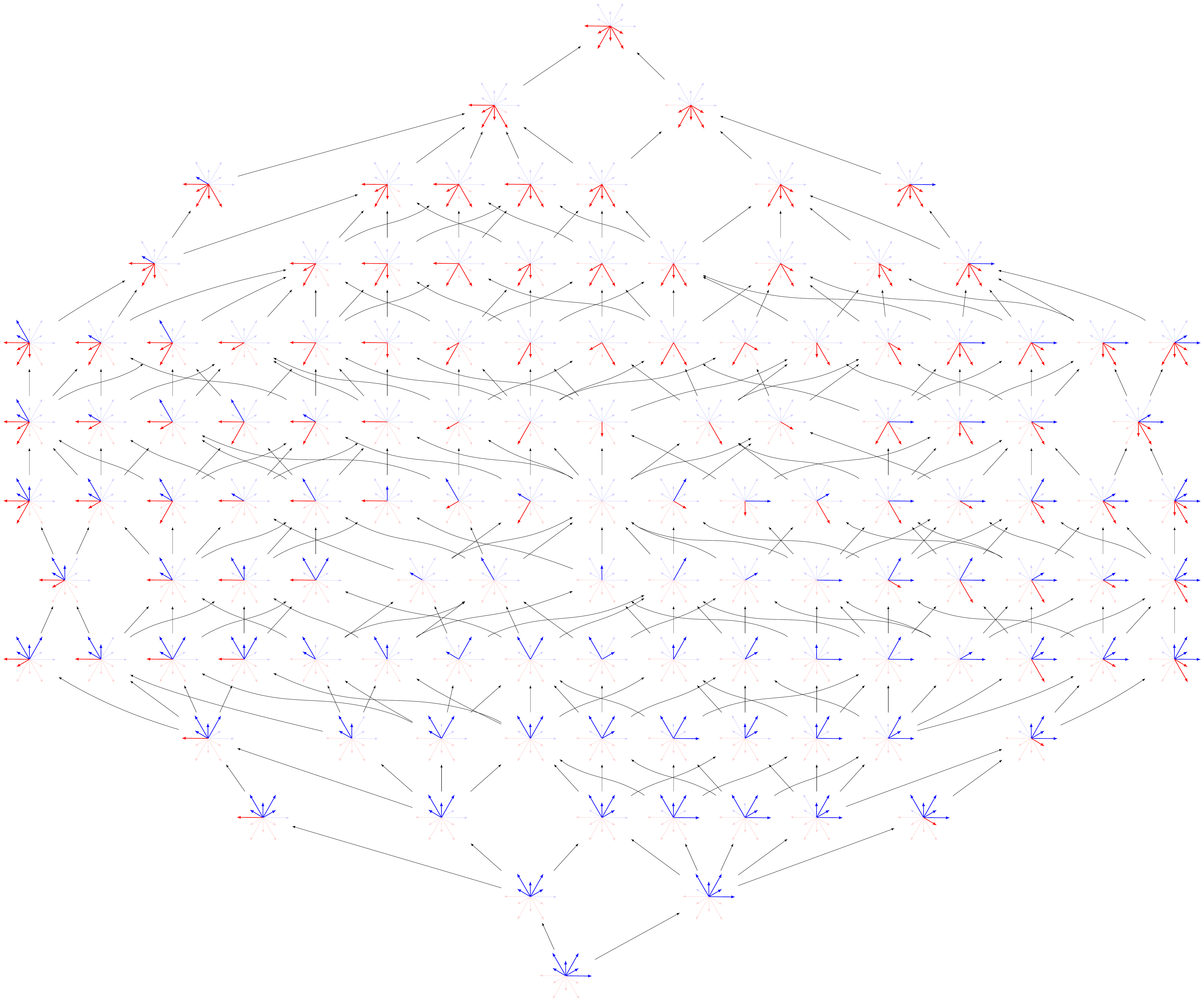}\vspace*{-.3cm}\end{minipage}}
{The weak order on $G_2$-posets.}
{fig:posetsG}

\begin{theorem}
\label{thm:posets}
The meet~$\meetC$ and the join~$\joinC$ both preserve antisymmetry.
Thus, when~$\Phi$ is crystallographic, the set~$\IPos(\Phi)$ of $\Phi$-posets induces a sublattice of the weak order on~$\IClo(\Phi)$.
\end{theorem}

\begin{proof}
Let~$\rel[R], \rel[S] \in \IPos(\Phi)$ and~$\rel[M] = \rel[R] \meetSC \rel[S]$ so that~$\ncd{\rel[M]} = {\rel[R]} \meetC {\rel[S]}$.
Assume that~$\ncd{\rel[M]}$ is not antisymmetric, and let~$\alpha \in (\ncd{\rel[M]})^+$ such that~$-\alpha \in (\ncd{\rel[M]})^-$.
Since~$(\ncd{\rel[M]})^- \subseteq \rel[M]^- = \rel[R]^- \cap \rel[S]^-$ and both~$\rel[R]$ and~$\rel[S]$ are antisymmetric, we obtain that~${\alpha \notin \rel[R]^+ \cup {\rel[S]}^+}$.
Since~$\alpha \in (\ncd{\rel[M]})^+ = \cl{(\rel[R]^+ \cup \rel[S]^+)}$, there exists~$\rel[X] \subseteq \rel[R]^+ \cup \rel[S]^+$ such that~$|{\rel[X]}| \ge 2$ and~$\alpha = \Sigma{\rel[X]}$.
By Theorem~\ref{thm:filtrationSummableSubsets}, there exists~$\beta \in \rel[X]$ such that~$\Sigma(\rel[X] \ssm \{\beta\}) \in \Phi$.
Since~$\rel[X] \ssm \{\beta\} \subseteq \rel[M]^+ \subseteq \ncd{\rel[M]}$, $-\alpha \in \ncd{\rel[M]}$ and~$\ncd{\rel[M]}$ is closed, we obtain that~$\Sigma(\rel[X] \ssm \{\beta\}) + (-\alpha) = -\beta \in (\ncd{\rel[M]})^- \subseteq \rel[R]^- \cap \rel[S]^-$.
As~$\beta \in \rel[R]^+ \cup \rel[S]^+$, this contradicts the antisymmetry of either~$\rel[R]$ or~$\rel[S]$.
\end{proof}

\begin{remark}
\label{rem:failLatticePosets}
Theorem~\ref{thm:posets} fails for non-crystallographic types.
An example in type~$H_3$ is given in Remark~\ref{rem:failLatticeClosed} (since the sets~$\rel[R], \rel[S], \rel[U]$ and~$\rel[V]$ are all antisymmetric and thus $\Phi$-posets).
\end{remark}

\begin{remark}
\label{rem:failLatticeConvexAntisym}
Even for crystallographic root systems, Theorem~\ref{thm:posets} fails for the other notions of closed sets.
An example in type~$B_3$ is given in Remark~\ref{rem:failLatticeConvex} (since the sets~$\rel[R], \rel[S], \rel[U]$ and~$\rel[V]$ are all antisymmetric and thus $\Phi$-posets).
\end{remark}

\begin{proposition}
\label{prop:coverRelationsPoset}
All cover relations in the weak order on~$\IPos(\Phi)$ are cover relations in the weak order on~$\IRel(\Phi)$.
In particular, the weak order on~$\IPos(\Phi)$ is still graded by $\rel[R] \mapsto |{\rel[R]^-}| - |{\rel[R]^+}|$. 
\end{proposition}

\begin{proof}
Consider a cover relation~$\rel \wole \rel[S]$ in the weak order on~$\IPos(\Phi)$.
We have~$\rel^+ \supseteq \rel[S]^+$ and~$\rel^- \subseteq \rel[S]^-$ where at least one of the inclusions is strict.
Suppose first that~$\rel^+ \supseteq \rel[S]^+$ and consider the set $\rel[X] \eqdef \set{\alpha \in \rel^+ \ssm \rel[S]^+}{\not\!\exists \; \beta, \gamma \in \rel^+ \text{ with } \alpha = \beta + \gamma}$.
This set~$\rel[X]$ is nonempty as it contains any~$\alpha$ in~$\rel^+ \ssm \rel[S]^+$ with~$\vert h \vert(\alpha)$ minimal. Consider now~$\alpha \in \rel[X]$ with~$\vert h\vert(\alpha)$ maximal and let~$\rel[T] \eqdef \rel \ssm \{\alpha\}$. We claim that~$\rel[T]$ is still a $\Phi$-poset. It is clearly still antisymmetric. For closedness, assume by means of contradiction that there is~$\beta, \gamma \in \rel[T]$ such that~$\alpha = \beta + \gamma$. Since~$\alpha \in \rel[X]\subseteq \Phi^+$, we can assume that~$\beta \in \rel^-$ and $\gamma \in \rel^+$, and we choose~$\beta$ so that~$\vert h \vert (\beta)$ is minimal. We claim that there is no~$\delta, \varepsilon \in \rel^+$ such that~$\gamma = \delta + \varepsilon$. Otherwise, since~$\alpha = \beta+ \gamma = \beta + \delta + \varepsilon \in \Phi$, we can assume by  Proposition~\ref{prop:twoSubsumsRoots} that $\beta + \delta \in \Phi \cup \{0\}$. 
If~$\beta + \delta \in \Phi^-$, then~$\beta + \delta \in \rel^-$ (since~$\rel$ is closed) which contradicts the minimality of $\beta$.
If~$\beta + \delta \in \Phi^+$, then~$\beta + \delta \in \rel^+$ (since~$\rel$ is closed), which together with~$\gamma \in \rel^+$ and~$(\beta + \delta) + \gamma = \alpha$ contradicts~$\alpha \in X$.
Finally, if~$\beta + \delta = 0$, then~$\beta = -\delta$ which contradicts the antisymmetry of $\rel$. 
This proves that there is no~$\delta, \varepsilon \in \rel^+$ such that~$\gamma = \delta + \varepsilon$. By maximality of~$h(\vert \alpha\vert)$ in our choice of~$\alpha$ this implies that~$\gamma \in \rel[S]$.
Since~$\beta \in \rel^- \subseteq \rel[S]^-$, we therefore obtain that~$\beta + \gamma = \alpha \notin \rel[S]$ while~$\beta, \gamma \in \rel[S]$, contradicting the closedness of~$\rel[S]$.
This proves that~$\rel[T]$ is closed and thus it is a $\Phi$-poset.
Moreover, we have~$\rel \ne \rel[T]$ and~$\rel \wole \rel[T] \wole \rel[S]$ where~$\rel[S]$ covers~$\rel$, which implies that~$\rel[S] = \rel[T] = \rel \ssm \{\alpha\}$.
We prove similarly that if~$\rel^- \ne \rel[S]^-$, there exists~$\alpha \in \Phi^-$ such that~$\rel[S] = {\rel} \cup \{\alpha\}$.
In both cases, $\rel \wole \rel[S]$ is a cover relation in the weak order on~$\IRel(\Phi)$.
\end{proof}

\begin{corollary}
\label{coro:coverRelationsPoset}
In the weak order on~$\IPos(\Phi)$, the $\Phi$-posets that cover a given $\Phi$-poset~$\rel \in \ISemiClo(\Phi)$ are precisely the relations:
\begin{itemize}
\item $\rel \ssm \{\alpha\}$ for any $\alpha \in \rel^+$ so that there is no~$\gamma, \delta \in \rel[R]^+$ with~$\alpha = \gamma + \delta$,
\item $\rel \cup \{\beta\}$, for any~$\beta \in \Phi^- \ssm \rel^-$ such that~$-\beta \notin \rel^+$ and $\beta + \gamma \in \Phi \implies \beta + \gamma \in \rel$ for all $\gamma \in \rel$.
\end{itemize}
\end{corollary}

\begin{remark}
We have gathered in Table~\ref{table:numerology} the number of $\Phi$-posets for the root systems of type~$A_n$, $B_n$, $C_n$ and~$D_n$ for small values of~$n$ (the other lines of the table will be explained in the next section).
Note that the number of semiclosed, closed and posets differ in type~$B_4$ and~$C_4$.
This should not come as a surprise since the notion of closed sets used in this paper (Definition~\ref{def:closed}) is not preserved when passing from roots to coroots.
\enlargethispage{1cm}

\renewcommand{\arraystretch}{1.1}
\begin{table}[h]
    \centerline{
    \begin{tabular}{c||l|l|l}
    type & $A$ & $B / C$ & $D$ ($n \ge 4$)\\
    \hline
    \# antisym. &
    	$3^1, 3^3, 3^6, 3^{10}$ [\href{https://oeis.org/A047656}{A047656}] &
		$3^1, 3^4, 3^9$ [\href{https://oeis.org/A060722}{A060722}] &
		$3^{12}$ [\href{https://oeis.org/A053764}{A053764}]
    	\\
    \# semiclosed &
    	$2^2, 7^2, 40^2, 357^2$ [\href{https://oeis.org/A006455}{A006455}] &
		$2^2, 12^2, 172^2, 5310^2  \, / \, 5318^2$ &
		$888^2$
    	\\
	\# closed &
		$4, 29, 355, 6942$ [\href{https://oeis.org/A000798}{A000798}] &
		$4, 55, 1785 \, / \, 1803$ &
		$18291$
		\\
    \# $\Phi$-posets &
    	$3, 19, 219, 4231$ [\href{https://oeis.org/A001035}{A001035}] &
    	$3, 37, 1235 \, / \, 1225$ &
    	$219$
    	\\
    \hline
    \# $\WOEP$ &
    	$2, 6, 24, 120$ [\href{https://oeis.org/A000142}{A000142}] &
    	$2, 8, 48, 384$ [\href{https://oeis.org/A000165}{A000165}] &
    	$192$ [\href{https://oeis.org/A002866}{A002866}]
    	\\
    \# $\WOIP$ &
    	$3, 17, 151, 1899$ [\href{https://oeis.org/A007767}{A007767}] &
    	$3, 27, 457$ &
    	$3959$
    	\\
    \# $\WOFP$ &
    	$3, 13, 75, 541$ [\href{https://oeis.org/A000670}{A000670}] &
    	$3, 17, 147, 1697$ [\href{https://oeis.org/A080253}{A080253}] &
    	$865$ [\href{https://oeis.org/A080254}{A080254}]
    	\\
    \hline
    \# $\COEP$ &
    	$2, 5, 14, 42$ [\href{https://oeis.org/A000108}{A000108}] &
    	$2, 6, 20, 70$ [\href{https://oeis.org/A000984}{A000984}] &
    	$50$ [\href{https://oeis.org/A051924}{A051924}]
    	\\
    \# $\COIP(\rm{bip})$ &
    	$3, 13, 70, 433$ &
    	$3, 18, 138, 1185$ &
    	$622$
		\\
    \# $\COIP(\rm{lin})$ &
    	$3, 13, 68, 399$ [\href{https://oeis.org/A000260}{A000260}] &
    	$3, 18, 132, 1069$ &
    	$578$
		\\
    \# $\COFP$ &
    	$3, 11, 45, 197$ [\href{https://oeis.org/A001003}{A001003}] &
    	$3, 13, 63, 321$ [\href{https://oeis.org/A001850}{A001850}] &
		$233$
    	\\
    \hline
    \# $\BOEP$ &
    	$2, 4, 8, 16, 32$ [\href{https://oeis.org/A000079}{A000079}] &
    	$2, 4, 8, 16, 32$ [\href{https://oeis.org/A000079}{A000079}] &
    	$16$ [\href{https://oeis.org/A000079}{A000079}]
    	\\
    \# $\BOIP$ &
    	$3, 9, 27, 81$ [\href{https://oeis.org/A000244}{A000244}] &
    	$3, 9, 27, 81$ [\href{https://oeis.org/A000244}{A000244}] &
    	$81$ [\href{https://oeis.org/A000244}{A000244}]
    	\\
    \# $\BOFP$ &
    	$3, 9, 27, 81$ [\href{https://oeis.org/A000244}{A000244}] &
    	$3, 9, 27, 81$ [\href{https://oeis.org/A000244}{A000244}] &
    	$81$ [\href{https://oeis.org/A000244}{A000244}]
    \end{tabular}
    }
    \caption{Numerology in types~$A_n$, $B_n$, $C_n$ and~$D_n$ for small values of~$n$. Further values can be found using the given references to~\cite{OEIS}.}
    \label{table:numerology}
\end{table}
\end{remark}
\renewcommand{\arraystretch}{1}

\section{Some relevant subposets}
\label{sec:examples}

In this section, we consider certain specific families of $\Phi$-posets corresponding to the vertices, the intervals and the faces in the permutahedron (Section~\ref{subsec:permutahedron}), the generalized associahedra (Section~\ref{subsec:associahedra}), and the cube (Section~\ref{subsec:cube}).

\subsection{Permutahedron}
\label{subsec:permutahedron}

The \defn{$W$-permutahedron}~$\Perm^p(W)$ is the convex hull of the orbit under~$W$ of a point~$p$ in the interior of the fundamental chamber of~$W$.
It has one vertex~$w(p)$ for each element~$w \in W$ and its graph is the Cayley graph of the set~$S$ of simple reflections of~$W$.
Moreover, when oriented in the linear direction~$w_\circ(p)-p$, its graph is the Hasse diagram of the \defn{weak order} on~$W$.
Recall that the weak order is defined equivalently for any~$v,w \in W$ by~$v \wole w$ if and only~if
\begin{itemize}
\item $\ell(v) + \ell(v^{-1}w) = \ell(w)$, where~$\ell(w)$ is the \defn{length} of~$w$, \ie the minimal length of an expression of the form~$\ell = s_1 \cdots s_k$ with~$s_1, \dots, s_k \in S$,
\item $v$ is a prefix of~$w$, \ie there exists~$u \in W$ such that~$w = vu$ and~$\ell(w) = \ell(v) + \ell(u)$,
\item $\inv(v) \subseteq \inv(w)$, where~$\inv$ denotes the \defn{inversion set}~$\inv(w) \eqdef \Phi^+ \cap w(\Phi^-)$,
\item there is an oriented path from~$v(p)$ to~$w(p)$ in the graph of the permutahedron oriented in the linear direction~$w_\circ(p)-p$.
\end{itemize}
In the sequel, we will often drop~$p$ from the notation~$\Perm^p(W)$ as the combinatorics of~$\Perm^p(W)$ is independent of~$p$ as long as this point is generic.

\subsubsection{Elements}

For an element~$w \in W$, we consider the $\Phi$-poset
\[
\rel(w) \eqdef w(\Phi^+).
\]
We say that~$\rel(w)$ is a \defn{weak order element poset} and let~$\WOEP(\Phi) \eqdef \set{\rel(w)}{w \in W}$ denote the collection of all such $\Phi$-posets.

\begin{remark}
Table~\ref{table:numerology} reports the cardinality of~$\WOEP(\Phi)$ in type~$A_n$, $B_n$, $C_n$ and~$D_n$ for small values of~$n$.
It is just the order of~$W$, which is known as the product formula
\[
|\WOEP(\Phi)| = |W| = \prod_{i \in [n]} d_i,
\]
where~$(d_1, \dots, d_n)$ are the degrees of~$W$.
\end{remark}

\begin{remark}
Geometrically, $\rel(w)$ is the set of roots of~$\Phi$ not contained in the cone of~$\Perm^p(W)$ at the vertex~$w(p)$, \ie $\rel(w) = \Phi \ssm \cone \set{w'(p) - w(p)}{w' \in W}$. See \fref{fig:facesPermutahedron}.
\end{remark}

We now characterize the $\Phi$-posets of~$\WOEP(\Phi)$.

\begin{proposition}
\label{prop:characterizationWOEP}
A $\Phi$-poset~$\rel \in \IPos(\Phi)$ is in~$\WOEP(\Phi)$ if and only if~$\alpha \in \rel$ or~$-\alpha \in \rel$~for~all~${\alpha \in \Phi}$.
\end{proposition}

\begin{proof}
This is folklore. See for instance~\cite[Chap.~6, 1.7, Coro.~1]{Bourbaki}.
\end{proof}

\begin{remark}
\label{rem:linearExtensions}
We have already encountered these $\Phi$-posets in Proposition~\ref{prop:maxExtension}: a poset is in~$\WOEP(\Phi)$ if and only if it is its unique extension.
In other words, the maximal extensions of a $\Phi$-poset~$\rel$ are all in~$\WOEP(\Phi)$, and it is thus natural to consider~$\linearExtensions(\rel) \eqdef \set{w \in W}{\rel \subseteq \rel(w)}$.
For example, in type~$A$, the set~$\linearExtensions(\rel)$ is the set of linear extensions of the poset~$\rel$.
\end{remark}

%
%

The following statement connects the subposet of the weak order induced by $\WOEP(\Phi)$ with the classical weak order on~$W$, and thus justifies the name in Definition~\ref{def:weakOrder}.

\begin{proposition}
\label{prop:weakOrderWOEP}
For~$w \in W$, we have~$\inv(w) = \Phi^+ \cap -\rel(w)$ and~${\rel(w) = \big( \Phi^+ \ssm \inv(w) \big) \sqcup -\inv(w)}$.
In particular, for~$v,w \in W$, we have~$\rel(v) \wole \rel(w)$ in the weak order on~$\WOEP(\Phi)$ if and only if~$v \wole w$ in the weak order on~$W$.
\end{proposition}

\begin{proof}
The first equality is just the definition of~$\inv(w)$ and the second comes from the fact that~$|\{\alpha, -\alpha\} \cap \rel(w)| = 1$, so that~$\rel(w)^- = \Phi^- \ssm -{\rel(w)}^+ = \Phi^- \ssm -\inv(w)$.
Finally, $v \wole w$ in the weak order on~$W$ $\iff$ $\inv(v) \subseteq \inv(w)$ $\iff$ $\Phi^+ \ssm \inv(v) \supseteq \Phi^+ \ssm \inv(w)$.
This shows the equivalence with~$\rel(v) \wole \rel(w)$.
\end{proof}

\begin{remark}
\label{rem:weakOrderWOEP}
In fact, $\rel(v) \wole \rel(w) \iff \rel(v)^+ \supseteq \rel(w)^+ \iff \rel(v)^- \subseteq \rel(w)^- \iff v \wole w$.
\end{remark}

\begin{corollary}
\label{coro:latticeWOEP}
The weak order on~$\WOEP(\Phi)$ is a lattice with meet and join
\[
\rel(v) \meetWOEP \rel(w) = \rel(v \meetWO w)
\qquad\text{and}\qquad
\rel(v) \joinWOEP \rel(w) = \rel(v \joinWO w).
\]
\end{corollary}

The following statement connects this lattice structure on~$\WOEP(\Phi)$ with that on~$\IPos(\Phi)$, and is our original motivation to study the weak order on~$\IPos(\Phi)$.

\begin{proposition}
\label{prop:sublatticeWOEP}
The set~$\WOEP(\Phi)$ induces a sublattice of the weak order on~$\IPos(\Phi)$.
\end{proposition}

\begin{proof}
Let~$\rel[R], \rel[S] \in \WOEP(\Phi)$ and~$\rel[M] = \rel[R] \meetSC \rel[S] = \cl{( \rel^+ \cup \rel[S]^+ )} \sqcup ( \rel^- \cap \rel[S]^- )$ so that~$\ncd{\rel[M]} = \rel[R] \meetC \rel[S]$.
Assume by means of contradiction that~$\ncd{\rel[M]}$ is not in~$\WOEP(\Phi)$, and consider~${\alpha \in \Phi^+}$ with~$|h|(\alpha)$ minimal such that~$\{\alpha, -\alpha\} \cap \ncd{\rel[M]} = \varnothing$.

Since~$(\ncd{\rel[M]})^+ = \rel[M]^+ = \cl{( \rel^+ \cup \rel[S]^+ )}$, we have~$\alpha \notin \rel[R]^+$ and~${\alpha \notin \rel[S]^+}$.
Since~$\rel[R], \rel[S] \in \WOEP(\Phi)$, we get~$-\alpha \in \rel[R]^-$ and~$-\alpha \in \rel[S]^-$, so that~${-\alpha \in \rel[M]^-}$.
Therefore ${-\alpha \in \rel[M] \ssm \ncd{\rel[M]}}$, so that there exists~$\rel[X] \subseteq \rel[M]^+$ such that~$\Sigma{\rel[X]}-\alpha \in \Phi \ssm \rel[M]$.
As usual, we assume that $\rel[X]$ has no vanishing subsum.
Since~${\rel[M]^+ = \cl{( \rel^+ \cup \rel[S]^+ )}}$, we can moreover assume that~$\rel[X] \subseteq \rel^+ \cup \rel[S]^+$ (up to developing each root of~$\rel[X]$).
We finally choose an inclusion minimal such subset~$\rel[X]$ of~$\rel^+ \cup \rel[S]^+$.

Assume first that~$\rel[X] = \{\beta\}$.
We have~$\beta \in \rel^+ \cup \rel[S]^+$, say~$\beta \in \rel^+$.
Since~${-\alpha \in \rel[M]^- = \rel^- \cap \rel[S]^-}$, $\beta \in \rel$ and~$\rel$ is closed, we have~$\beta-\alpha \in \rel$.
Since~$\beta-\alpha \notin \rel[M]^+$, we obtain that~$\beta-\alpha \in \Phi^-$.
Therefore, as~$\beta \in \Phi^+$, we have $|h|(\beta-\alpha) < |h|(\alpha)$.
By minimality of~$|h|(\alpha)$, we obtain that~$\alpha-\beta \in \ncd{\rel[M]}$.
We conclude that~$\alpha-\beta \in \ncd{\rel[M]}$ and $\beta \in \rel^+ \subseteq \ncd{\rel[M]}$ while~$\alpha \notin \ncd{\rel[M]}$, contradicting the closedness of~$\ncd{\rel[M]}$.

Assume now that~$|{\rel[X]}| \ge 2$.
Since ${\alpha \notin \rel[M]^+ = \cl{( \rel^+ \cup \rel[S]^+ )}}$ and~$\rel[X] \subseteq \rel^+ \cup \rel[S]^+$, we obtain that~${\rel[X] \cup \{-\alpha\}}$ has no vanishing subsums.
Therefore, Proposition~\ref{thm:manySummableSubsets} ensures that~$\rel[X] \cup \{-\alpha\}$ has at least two strict summable subsets.
In particular, there is~${\rel[Y]} \subsetneq {\rel[X]}$ such that~$\Sigma{\rel[Y]}-\alpha \in \Phi$.
By minimality of~$\rel[X]$, we obtain that~$\Sigma{\rel[Y]}-\alpha \in \rel[M]$.
We distinguish two cases:
\begin{itemize}
\item If~$\Sigma{\rel[Y]}-\alpha \in \rel[M]^+$, then~$\Sigma{\rel[X]}-\alpha \notin \rel[M]^+$ while~$\Sigma{\rel[Y]}-\alpha \in \rel[M]^+$ and~$\rel[X] \ssm \rel[Y] \subseteq \rel[M]^+$ contradicts the closedness of~$\rel[M]^+$.
\item If~$\Sigma{\rel[Y]}-\alpha \in \rel[M]^-$, then~$|h|(\Sigma{\rel[Y]}-\alpha) < |h|(\alpha)$. By minimality of~$|h|(\alpha)$, we obtain that 
	\begin{itemize}
	\item either~$\Sigma{\rel[Y]}-\alpha \in \ncd{\rel[M]}$, which implies that~$\Sigma{\rel[X]}-\alpha = (\Sigma{\rel[Y]}-\alpha) + \big( \Sigma(\rel[X]\ssm\rel[Y]) \big) \in \rel[M]$, a contradiction.
	\item or~$\alpha-\Sigma{\rel[Y]} \in \ncd{\rel[M]}$, which implies that~$\alpha = (\alpha-\Sigma{\rel[Y]}) + \Sigma{\rel[Y]} \in \ncd{\rel[M]}$, which contradict our assumption on~$\alpha$.
	\end{itemize}
\end{itemize}
As we reached a contradiction in all cases, we conclude that~$\ncd{\rel[M]} \in \WOEP(\Phi)$.
The proof is similar for the join.
\end{proof}

\subsubsection{Intervals}

For~$w,w' \in W$ with~$w \wole w'$, we denote by~$[w,w'] \eqdef \set{v \in W}{w \wole v \wole w'}$ the \defn{weak order interval} between~$w$ and~$w'$.
We associate to each weak order interval~$[w,w']$ the $\Phi$-poset
\[
\rel(w,w') \eqdef \bigcap_{v \in [w,w']} \rel(v) = \rel(w) \cap \rel(w') = \rel(w)^- \sqcup \rel(w')^+.
\]
Say that~$\rel(w,w')$ is a \defn{weak order interval poset} and let~$\WOIP(\Phi) \eqdef \set{\rel(w,w')}{w,w' \in W, \, w \wole w'}$ denote the collection of all such $\Phi$-posets.
Table~\ref{table:numerology} reports the cardinality of~$\WOIP(\Phi)$ in type~$A_n$, $B_n$, $C_n$ and~$D_n$ for small values of~$n$.

Recall from Remark~\ref{rem:linearExtensions} that we denote by~$\linearExtensions(\rel) \eqdef \set{w \in W}{\rel \subseteq \rel(w)}$ the set of maximal extensions of a $\Phi$-poset~$\rel$.
We will use the following observation to characterize these $\Phi$-posets.

\begin{lemma}
\label{lem:characterizationWOIP}
A $\Phi$-poset~$\rel \in \IPos(\Phi)$ is in~$\WOIP(\Phi)$ if and only if $\linearExtensions(\rel)$ has a unique weak order minimum~$w$ (resp.~maximum~$w'$) that moreover satisfies~$\rel(w)^- = \rel^-$ (resp.~$\rel(w')^+ = \rel^+$).
\end{lemma}

\begin{proof}
Observe first that Remark~\ref{rem:weakOrderWOEP} implies that~$\rel(w,w') \subseteq \rel(v) \iff \rel(w)^- \subseteq \rel(v)^-$ and ${\rel(w')^+ \subseteq \rel(v)^+ \iff v \in [w,w']}$. Therefore, $\linearExtensions \big( \rel(w,w') \big)$ has a unique weak order minimum~$w$ and a unique weak order maximum~$w'$ and~$\rel(w)^- = \rel(w,w')^-$ while~$\rel(w')^+ = \rel(w,w')^+$.

Conversely, if~$\linearExtensions(\rel)$ has a unique weak order minimum~$w$ and a unique weak order maximum~$w'$ and~$\rel(w)^- = \rel^-$ while~$\rel(w')^+ = \rel^+$, then~$\rel = \rel(w)^- \sqcup \rel(w')^+ = \rel(w,w')$ by definition.
\end{proof}

\begin{remark}
In Lemma~\ref{lem:characterizationWOIP}, the final hypothesis is crucial as it may happen that~$\rel \ne \bigcap \extensions(\rel)$ (consider for example~$\rel = \{\alpha_1+\alpha_2, \alpha_2\}$ in type~$B_2$).
\end{remark}

We can now characterize the $\Phi$-posets of~$\WOIP(\Phi)$.

\begin{proposition}
\label{prop:characterizationWOIP}
A $\Phi$-poset~$\rel \in \IPos(\Phi)$ is in~$\WOIP(\Phi)$ if and only if $\alpha + \beta \in \rel$ implies $\alpha \in \rel$ or $\beta \in \rel$ for all~$\alpha, \beta \in \Phi^-$ and all~$\alpha, \beta \in \Phi^+$.
\end{proposition}

\begin{proof}
By Lemma~\ref{lem:characterizationWOIP}, this boils down to show that the following assertions are equivalent:
\begin{enumerate}[(i)]
\item $\linearExtensions(\rel)$ has a unique weak order minimum~$w$ (resp.~maximum~$w'$) that moreover satisfies $\rel(w)^- = \rel^-$ (resp.~$\rel(w')^+ = \rel^+$),
\item $\alpha + \beta \in \rel$ implies $\alpha \in \rel$ or $\beta \in \rel$ for all~$\alpha, \beta \in \Phi^-$ (resp.~for all~$\alpha, \beta \in \Phi^+$).
\end{enumerate}
We prove the result for the maximum and~$\alpha, \beta \in \Phi^+$.
The result for the minimum and~$\alpha, \beta \in \Phi^-$ follows by symmetry.

Assume first that~(ii) holds.
Consider the subset of roots~$\rel[S] \eqdef \rel^+ \cup (\Phi^- \ssm -\rel^+)$.
Note that~$\rel[R] \subseteq \rel[S]$ (since~$\rel[R]$ is antisymmetric), that~$\rel[S]$ is antisymmetric, and that~$\rel[T] \wole \rel[S]$ for any antisymmetric~$\rel[T]$ such that~$\rel[R] \subseteq \rel[T]$ (as~$\rel[R]$ has been completed with all possible negative roots to obtain~$\rel[S]$).
We moreover claim that~$\rel[S]$ is closed.
Indeed, consider~$\alpha, \beta \in \rel[S]$ such that~$\alpha + \beta \in \Phi$. We distinguish four cases:
\begin{itemize}
\item If~$\alpha \in \rel$ and~$\beta \in \rel$, then~$\alpha + \beta \in \rel \subseteq \rel[S]$.
\item If~$\alpha \notin \rel$ and~$\beta \in \rel$, then~$\alpha \in \rel[S] \ssm \rel[R] \subseteq \Phi^-$ so that~$-\alpha \in \Phi^+ \ssm \rel^+$. Then,
	\begin{itemize}
	\item if~$\alpha + \beta \in \Phi^+$, then we have~$-\alpha \in \Phi^+ \ssm \rel^+$ and $\alpha + \beta \in \Phi^+$ with~$-\alpha + (\alpha + \beta) = \beta \in \rel$ so that Condition~(ii) ensures that~$\alpha + \beta \in \rel$,
	\item if~$\alpha + \beta \in \Phi^-$, then $-(\alpha+\beta) \notin \rel$ (as otherwise we would have~$-\alpha = -(\alpha+\beta) + \beta \in \rel$, a contradiction). Therefore, $\alpha + \beta \in \Phi^- \ssm -\rel[R]^+ \subseteq \rel[S]$.
	\end{itemize}
\item If~$\alpha \in \rel$ and~$\beta \notin \rel$, the argument is symmetric.
\item If~$\alpha \notin \rel$ and~$\beta \notin \rel$, then~$\alpha, \beta \in \rel[S] \ssm \rel[R] \subseteq \Phi^-$ and~$-\alpha, -\beta \in \Phi^+ \ssm \rel$. By condition~(ii), this implies that~$-\alpha-\beta \in \Phi^+ \ssm \rel^+$. Therefore, $\alpha + \beta \in \Phi^- \ssm -\rel^+ \subseteq \rel[S]$.
\end{itemize}
We thus obtained in all cases that~$\alpha + \beta \in \rel[S]$ so that~$\rel[S]$ is closed.
We conclude that $\rel[S]$ is a $\Phi$-poset and that~$\rel[T] \wole \rel[S]$ for any antisymmetric~$\rel[T]$ such that~$\rel[R] \subseteq \rel[T]$.
In particular, $\rel[S]$ is the unique maximum of the set~$\extensions(\rel)$ of extensions of~$\rel$.
Moreover, $\rel[S]^+ = \rel[R]^+$.
Using Propositions~\ref{prop:maxExtension} and~\ref{prop:characterizationWOEP}, we obtain that there exist~$w' \in W$ such that~$\rel[S] = \rel(w')$.
This concludes the proof that~(ii)$\implies$(i).

Conversely, assume by means of contradiction that~(i) holds but not~(ii).
Let~$w'$ denote the weak order maximal element of~$\linearExtensions(\rel)$, and let~$\alpha, \beta \in \Phi^+ \ssm \rel$ be such that~$\alpha + \beta \in \rel$.
We then distinguish two cases:
\begin{itemize}
\item If~$\alpha \in \rel(v)$ for all~$v \in \linearExtensions(\rel)$, then~$\alpha \in \rel(w')^+ = \rel^+$. Contradiction.
\item Otherwise, there exists~$v \in \linearExtensions(\rel)$ such that~$-\alpha \in \rel(v)$. Since~$v \wole w'$, this gives~${-\alpha \in \rel(w')}$. Since~$\alpha + \beta \in \rel \subseteq \rel(w')$ and~$\rel(w')$ is closed, we get~${\beta \in \rel(w')^+ = \rel^+}$. Contradiction.
\qedhere
\end{itemize}
\end{proof}

We now describe the weak order on~$\WOIP(\Phi)$.
It corresponds to the Cartesian product order on intervals of the weak order.

\begin{proposition}
\label{prop:weakOrderWOIP}
For any two weak order intervals~$v \wole v'$ and~$w \wole w'$, we have~$\rel(v,v') \wole \rel(w,w')$ in the weak order on~$\WOIP(\Phi)$ if and only if~$v \wole w$ and~$v' \wole w'$.
\end{proposition}

\begin{proof}
From the definition of~$\rel(w,w')$ and Remark~\ref{rem:weakOrderWOEP}, we have
\[
\begin{array}[b]{c@{\quad\iff\quad}c@{\quad\text{and}\quad}c}
\rel(v,v') \wole \rel(w,w')
& \rel(v,v')^+ \supseteq \rel(w,w')^+ & \rel(v,v')^- \subseteq \rel(w,w')^- \\
& \rel(v')^+ \supseteq \rel(w')^+ & \rel(v)^- \subseteq \rel(w)^- \\
& v' \wole w' & v \wole w.
\end{array}
\qedhere
\]
\end{proof}

\begin{corollary}
\label{coro:latticeWOIP}
The weak order on~$\WOIP(\Phi)$ is a lattice with meet and join
\[
\rel(v,v') \meetWOIP \rel(w,w') = \rel(v \meetWO w, v' \meetWO w')
\quad\text{and}\quad
\rel(v,v') \joinWOIP \rel(w,w') = \rel(v \joinWO w, v' \joinWO w').
\]
\end{corollary}

\begin{remark}
It follows from the expressions of~$\meetWOIP$ and~$\joinWOIP$ that~$\WOEP(\Phi)$ also induces a sublattice of~$\WOIP(\Phi)$.
\end{remark}

\begin{remark}
To conclude on intervals, we however observe that the weak order on~$\WOIP(\Phi)$ is not a sublattice of the weak order on $\Phi$-posets.
For example, in type~$A_2$ we have
\[
\{\alpha_1, \alpha_1+\alpha_2\} \joinC \{\alpha_2, \alpha_1+\alpha_2\} = \{\alpha_1+\alpha_2\}
\quad\text{while}\quad
\{\alpha_1, \alpha_1+\alpha_2\} \joinWOIP \{\alpha_2, \alpha_1+\alpha_2\} = \varnothing.
\]
\end{remark}

\subsubsection{Faces}
\label{subsubsec:WOFP}

The faces of the permutohedron~$\Perm^p(W)$ correspond to the cosets of the standard parabolic subgroups of~$W$.
Recall that a \defn{standard parabolic subgroup} of~$W$ is a subgroup~$W_I$ generated by a subset~$I$ of the simple reflections of~$W$.
Its simple roots are the simple roots~$\Delta_I$ of~$\Delta$ corresponding to~$I$, its root system is~$\Phi_I = W_I(\Delta_I) = \Phi \cap \R\Delta_I$ and its longest element is denoted by~$w_{\circ,I}$.
A \defn{standard parabolic coset} is a coset under the action of a standard parabolic subgroup~$W_I$.
Such a standard parabolic coset can be written as~$xW_I$ where~$x$ is its minimal length coset representative (thus $x$ has no descent in~$I$, see Section~\ref{subsec:cube}).
Each standard parabolic coset~$xW_I$ (with~$I \subseteq S$ disjoint from the descent set~$\des(x)$ of~$x$) corresponds to a face
\[
\face(xW_I) = x \big( \Perm^p(W_I) \big) = \Perm^{x(p)} \big( xW_Ix^{-1} \big).
\]
See \fref{fig:facesPermutahedron} for an illustration in type~$A_2$ and~$B_2$.

In~\cite{DermenjianHohlwegPilaud}, A.~Dermenjian, C.~Hohlweg and V.~Pilaud also associated to each standard parabolic coset~$xW_I$ the set of roots~$\overline{\rel}(xW_I) \eqdef x(\Phi^- \cup \Phi_I^+)$.
These $\Phi$-posets were characterized in~\cite{DermenjianHohlwegPilaud} as follows.

\begin{proposition}[{\cite[Coro.~3.9]{DermenjianHohlwegPilaud}}]
\label{prop:characterizationBarWOPF}
The following assertions are equivalent for a subset of roots~$\rel \in \IRel(\Phi)$:
\begin{enumerate}[(i)]
\item $\rel = \overline{\rel}(xW_I)$ for some parabolic coset~$xW_I$ of~$W$,
\item $\rel = \set{\alpha \in \Phi}{\psi(\alpha) \ge 0}$ for some linear function~$\psi : V \to \R$,
\item $\rel = \Phi \cap \cone(\rel)$ is convex closed and $|{\rel} \cap \{\alpha, -\alpha\}| \ge 1$ for all~$\alpha \in \Phi$.
\end{enumerate}
\end{proposition}

Moreover, they used this definition to recover the following order on faces of the permutahedron, defined initially in type~$A$ in~\cite{KrobLatapyNovelliPhanSchwer} and latter for arbitrary finite Coxeter groups in~\cite{PalaciosRonco}.

\begin{proposition}[\cite{DermenjianHohlwegPilaud}]
\label{prop:facialWeakOrder}
The following assertions are equivalent for two standard parabolic cosets~$xW_I = [x, xw_{\circ,I}]$ and~$yW_J = [y, yw_{\circ,J}]$ of~$W$:
\begin{itemize}
\item $x \wole y$ and~$x w_{\circ,I} \wole y w_{\circ,J}$,
\item $\overline{\rel}(xW_I)^+ \subseteq \overline{\rel}(yW_J)^+$ and $\overline{\rel}(xW_I)^- \supseteq \overline{\rel}(yW_J)^-$,
\item $xW_I \wole yW_J$ for the transitive closure~$\wole$ of the two cover relations~$xW_I \prec xW_{I \cup \{s\}}$ for~$s \notin I \cup \des(x)$ and $xW_I \prec (xw_{\circ, I}w_{\circ, I \ssm \{s\}})W_{I \ssm \{s\}}$ for $s \in I$.
\end{itemize}
The resulting order on standard parabolic cosets is the \defn{facial weak order} defined in~\cite{KrobLatapyNovelliPhanSchwer, PalaciosRonco, DermenjianHohlwegPilaud}.
This order extends the weak order on~$W$ since~$xW_\varnothing \wole yW_\varnothing \iff x \wole y$ for any~${x,y \in W}$.
Moreover, it defines a lattice on standard parabolic cosets of~$W$ with meet and join
\begin{gather*}
xW_I \meetFWO yW_J = \zm W_\Km \quad\text{where}\quad \zm = x \meetWO y \quad\text{and}\quad \Km = \des \big( \zm^{-1} (xw_{\circ,I} \meetWO yw_{\circ,J}) \big), \\
xW_I \joinFWO yW_J = \zj  W_\Kj \quad\text{where}\quad \zj = xw_{\circ,I} \joinWO yw_{\circ,J} \quad\text{and}\quad \Kj = \des \big( \zj^{-1} (x \joinWO y) \big).
\end{gather*}
\end{proposition}

Note that~$\overline{\rel}(xW_I)$ is not a $\Phi$-poset as it is not antisymmetric when~$I \ne \varnothing$.
Here, we will therefore associate to~$xW_I$ the set of roots
\[
\rel(xW_I) \eqdef \Phi \ssm \overline{\rel}(xW_I) = x(\Phi^+ \ssm \Phi_I^+).
\]
Note that $\rel(xW_I)$ coincides with the weak order interval poset~$\rel(x, xw_{\circ,I})$.
We say that~$\rel(xW_I)$ is a \defn{weak order face poset} and we let~$\WOFP(\Phi) \eqdef \set{\rel(xW_I)}{xW_I \text{ standard parabolic coset of } W}$ denote the collection of all such $\Phi$-posets.
Table~\ref{table:numerology} reports the cardinality of~$\WOFP(\Phi)$ in type~$A_n$, $B_n$, $C_n$ and~$D_n$ for small values of~$n$.

\begin{remark}
Geometrically, $\rel(xW_I)$ is the set of roots of~$\Phi$ not contained in the cone of~$\Perm^p(W)$ at the face~$\face(xW_I)$, \ie $\rel(xW_I) = \Phi \ssm \cone \set{w'(p) - w(p)}{w \in xW_I, w' \in W}$. See \fref{fig:facesPermutahedron}.

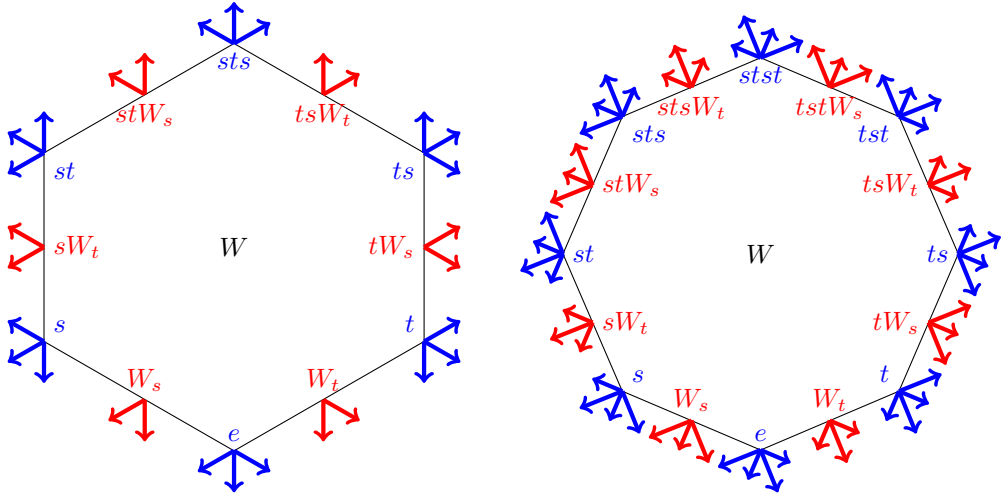
\begin{figure}[t]
\DeclareDocumentCommand{\rs}{ O{1.1cm} O{->} m m O{0}} {
	\def \radius {#1}
	\def \inputPoints{#3}
	\def \includeRoots{#4}
	\def \style {#2}
	\def \initialRotation {#5}
	
	\pgfmathtruncatemacro{\points}{\inputPoints * 2}
	\pgfmathsetmacro{\degrees}{360 / \points}
	
	\coordinate (0) at (0,0);
	
	\foreach \x in {1,...,\points}{%
		\pgfmathsetmacro{\location}{(\points+(\x-1))*\degrees + \initialRotation}
		\ifthenelse{\equal{\inputPoints}{3}}{
			\coordinate (\x) at (\location:\radius);
		}{
			\coordinate (\x) at (\location:1.5*\radius*(1 + 0.41*Mod(\x,2);));		
		}
	}

	\ifthenelse{\equal{\includeRoots}{}}{
	}{
		\foreach \x in {1,...,\points}{%
			\edef \showPoint {0};

			\foreach \y in \includeRoots {
				\ifthenelse{\equal{\x}{\y}}{
					\xdef \showPoint {1};
				}{}
			}
			
			\ifthenelse{\equal{\showPoint}{1}}{
				\draw[->, ultra thick] (0) -- (\x);
			}{}
		}
	}  
}

\centerline{
	\begin{tikzpicture}
		[scale=2.5,
		bface/.style={color=blue},
		rface/.style={color=red},
		gface/.style={color=black}
		]
		%
		\coordinate (e) at (0,0.42);
		\coordinate (s) at (-1,1);
		\coordinate (t) at (1,1);
		\coordinate (st) at (-1,2);
		\coordinate (ts) at (1,2);
		\coordinate (sts) at (0,2.58);
		%
		\coordinate (Ws) at (-0.47,0.69);
		\coordinate (Wt) at (0.47,0.69);
		\coordinate (tWs) at (1,1.5);
		\coordinate (sWt) at (-1,1.5);
		\coordinate (stWs) at (-0.47,2.31);
		\coordinate (tsWt) at (0.47,2.31);
		%
		\coordinate (W) at (0,1.5);
		%
		\draw (e) -- (s);
		\draw (e) -- (t);
		\draw (s) -- (st);
		\draw (t) -- (ts);
		\draw (st) -- (sts);
		\draw (ts) -- (sts);
		%
		\begin{scope}[shift={(e)}, scale=0.25, bface]
			\node[bface, above] {$e$};
			\begin{scope}[scale=0.8]
				\rs{3}{1,2,3}[210]
			\end{scope}
		\end{scope}
		\begin{scope}[shift={(s)}, scale=0.25, bface]
			\node[bface, above right] {$s$};
			\begin{scope}[scale=0.8]
				\rs{3}{1,2,6}[210]
			\end{scope}
		\end{scope}
		\begin{scope}[shift={(t)}, scale=0.25, bface]
			\node[bface, above left] {$t$};
			\begin{scope}[scale=0.8]
				\rs{3}{2,3,4}[210]
			\end{scope}
		\end{scope}
		\begin{scope}[shift={(st)}, scale=0.25, bface]
			\node[bface, below right] {$st$};
			\begin{scope}[scale=0.8]
				\rs{3}{1,5,6}[210]
			\end{scope}
		\end{scope}
		\begin{scope}[shift={(ts)}, scale=0.25, bface]
			\node[bface, below left] {$ts$};
			\begin{scope}[scale=0.8]
				\rs{3}{3,4,5}[210]
			\end{scope}
		\end{scope}
		\begin{scope}[shift={(sts)}, scale=0.25, bface]
			\node[bface, below] {$sts$};
			\begin{scope}[scale=0.8]
				\rs{3}{4,5,6}[210]
			\end{scope}
		\end{scope}
		%
		\begin{scope}[shift={(Ws)}, scale=0.25,rface]
			\node[rface, above] {$W_{s}$};
			\begin{scope}[scale=0.8]
				\rs{3}{1,2}[210]
			\end{scope}
		\end{scope}
		\begin{scope}[shift={(Wt)}, scale=0.25, rface]
			\node[rface, above] {$W_{t}$};
			\begin{scope}[scale=0.8]
				\rs{3}{2,3}[210]
			\end{scope}
		\end{scope}
		\begin{scope}[shift={(sWt)}, scale=0.25, rface]
			\node[rface, right] {$sW_{t}$};
			\begin{scope}[scale=0.8]
				\rs{3}{1,6}[210]
			\end{scope}
		\end{scope}
		\begin{scope}[shift={(tWs)}, scale=0.25, rface]
			\node[rface, left] {$tW_{s}$};
			\begin{scope}[scale=0.8]
				\rs{3}{3,4}[210]
			\end{scope}
		\end{scope}
		\begin{scope}[shift={(stWs)}, scale=0.25, rface]
			\node[rface, below] {$stW_{s}$};
			\begin{scope}[scale=0.8]
				\rs{3}{5,6}[210]
			\end{scope}
		\end{scope}
		\begin{scope}[shift={(tsWt)}, scale=0.25, rface]
			\node[rface, below] {$tsW_{t}$};
			\begin{scope}[scale=0.8]
				\rs{3}{4,5}[210]
			\end{scope}
		\end{scope}
		%
		\begin{scope}[shift={(W)}, scale=0.25, gface]
			\node[gface] {$W$};
			\begin{scope}[scale=0.8]
				\rs{3}{}[210]
			\end{scope}
		\end{scope}
	\end{tikzpicture}
	\hspace{.5cm}
	\begin{tikzpicture}
		[scale=1.3,
		bface/.style={color=blue},
		rface/.style={color=red},
		gface/.style={color=black}
		]
		%
		\coordinate (e) at (0,0);
		\coordinate (s) at (-1.4,0.6);
		\coordinate (t) at (1.4,0.6);
		\coordinate (st) at (-2,2);
		\coordinate (ts) at (2,2);
		\coordinate (sts) at (-1.4,3.4);
		\coordinate (tst) at (1.4,3.4);
		\coordinate (stst) at (0,4);
		%
		\coordinate (Ws) at (-0.7,0.3);
		\coordinate (Wt) at (0.7,0.3);
		\coordinate (tWs) at (1.7,1.3);
		\coordinate (sWt) at (-1.7,1.3);
		\coordinate (stWs) at (-1.7,2.7);
		\coordinate (tsWt) at (1.7,2.7);
		\coordinate (tstWs) at (0.7,3.7);
		\coordinate (stsWt) at (-0.7,3.7);
		%
		\coordinate (W) at (0,2);
		%
		\draw (e) -- (s);
		\draw (e) -- (t);
		\draw (s) -- (st);
		\draw (t) -- (ts);
		\draw (st) -- (sts);
		\draw (ts) -- (tst);
		\draw (tst) -- (stst);
		\draw (sts) -- (stst);
		%
		\begin{scope}[shift={(e)}, scale=0.25, bface]
			\node[bface, above] {$e$};
			\begin{scope}[scale=0.8]
				\rs{4}{1,2,3,4}[202]
			\end{scope}
		\end{scope}
		\begin{scope}[shift={(s)}, scale=0.25, bface]
			\node[bface, above right] {$s$};
			\begin{scope}[scale=0.8]
				\rs{4}{1,2,3,8}[202]
			\end{scope}
		\end{scope}
		\begin{scope}[shift={(st)}, scale=0.25, bface]
			\node[bface, right] {$st$};
			\begin{scope}[scale=0.8]
				\rs{4}{1,2,7,8}[202]
			\end{scope}
		\end{scope}
		\begin{scope}[shift={(sts)}, scale=0.25, bface]
			\node[bface, below right] {$sts$};
			\begin{scope}[scale=0.8]
				\rs{4}{1,6,7,8}[202]
			\end{scope}
		\end{scope}
		\begin{scope}[shift={(stst)}, scale=0.25, bface]
			\node[bface, below] {$stst$};
			\begin{scope}[scale=0.8]
				\rs{4}{5,6,7,8}[202]
			\end{scope}
		\end{scope}
		\begin{scope}[shift={(tst)}, scale=0.25, bface]
			\node[bface, below left] {$tst$};
			\begin{scope}[scale=0.8]
				\rs{4}{4,5,6,7}[202]
			\end{scope}
		\end{scope}
		\begin{scope}[shift={(ts)}, scale=0.25, bface]
			\node[bface, left] {$ts$};
			\begin{scope}[scale=0.8]
				\rs{4}{3,4,5,6}[202]
			\end{scope}
		\end{scope}
		\begin{scope}[shift={(t)}, scale=0.25, bface]
			\node[bface, above left] {$t$};
			\begin{scope}[scale=0.8]
				\rs{4}{2,3,4,5}[202]
			\end{scope}
		\end{scope}
		%
		\begin{scope}[shift={(Ws)}, scale=0.25, rface]
			\node[rface, above] {$W_{s}$};
			\begin{scope}[scale=0.8]
				\rs{4}{1,2,3}[202]
			\end{scope}
		\end{scope}
		\begin{scope}[shift={(sWt)}, scale=0.25, rface]
			\node[rface, right] {$sW_t$};
			\begin{scope}[scale=0.8]
				\rs{4}{1,2,8}[202]
			\end{scope}
		\end{scope}
		\begin{scope}[shift={(stWs)}, scale=0.25, rface]
			\node[rface, right] {$stW_s$};
			\begin{scope}[scale=0.8]
				\rs{4}{1,7,8}[202]
			\end{scope}
		\end{scope}
		\begin{scope}[shift={(stsWt)}, scale=0.25, rface]
			\node[rface, below] {$stsW_t$};
			\begin{scope}[scale=0.8]
				\rs{4}{6,7,8}[202]
			\end{scope}
		\end{scope}
		\begin{scope}[shift={(tstWs)}, scale=0.25, rface]
			\node[rface, below] {$tstW_s$};
			\begin{scope}[scale=0.8]
				\rs{4}{5,6,7}[202]
			\end{scope}
		\end{scope}
		\begin{scope}[shift={(tsWt)}, scale=0.25, rface]
			\node[rface, left] {$tsW_t$};
			\begin{scope}[scale=0.8]
				\rs{4}{4,5,6}[202]
			\end{scope}
		\end{scope}
		\begin{scope}[shift={(tWs)}, scale=0.25, rface]
			\node[rface, left] {$tW_s$};
			\begin{scope}[scale=0.8]
				\rs{4}{3,4,5}[202]
			\end{scope}
		\end{scope}
		\begin{scope}[shift={(Wt)}, scale=0.25, rface]
			\node[rface, above] {$W_t$};
			\begin{scope}[scale=0.8]
				\rs{4}{2,3,4}[202]
			\end{scope}
		\end{scope}
		\begin{scope}[shift={(W)}, scale=0.25, gface]
			\node[gface] {$W$};
			\begin{scope}[scale=0.8]
				\rs{4}{}[202]
			\end{scope}
		\end{scope}
	\end{tikzpicture}
}
	\caption{The sets~$\rel(xW_I)$ of the standard parabolic cosets~$xW_I$ in type~$A_2$~(left) and~$B_2$ (right). Note that positive roots point downwards.}
	\label{fig:facesPermutahedron}
\end{figure}
\end{remark}

Proposition~\ref{prop:characterizationBarWOPF} yields the following characterization of the $\Phi$-posets in~$\WOFP(\Phi)$.

\begin{proposition}
\label{prop:characterizationWOFP}
The following assertions are equivalent for a subset of roots~$\rel \in \IRel(\Phi)$:
\begin{enumerate}[(i)]
\item $\rel$ is a weak order face poset of~$\WOFP(\Phi)$,
\item $\rel = \set{\alpha \in \Phi}{\psi(\alpha) < 0}$ for some linear function~$\psi : V \to \R$,
\item $\rel = \Phi \cap \cone(\rel)$ is convex closed and $|{\rel} \cap \{\alpha, -\alpha\}| \le 1$ for all~$\alpha \in \Phi$.
\end{enumerate}
\end{proposition}

\begin{proof}
This immediately follows from the characterization of~$\overline{\rel}(xW_I)$ in Proposition~\ref{prop:characterizationBarWOPF} and the definition~$\rel(xW_I) \eqdef \Phi \ssm \overline{\rel}(xW_I)$.
\end{proof}

We now observe that the weak order induced by~$\WOFP(\Phi)$ corresponds to the facial weak order of~\cite{PalaciosRonco, DermenjianHohlwegPilaud}.

\begin{proposition}
\label{prop:weakOrderWOFP}
For any standard parabolic cosets~$xW_I$ and~$yW_J$, we have~$\rel(xW_I) \wole \rel(yW_J)$ in the weak order on~$\WOFP(\Phi)$ if and only if~$xW_I \wole yW_J$ in facial weak order.
\end{proposition}

\begin{proof}
By definition of~$\rel(xW_I)$ and Proposition~\ref{prop:facialWeakOrder}, we have
\begin{align*}
\rel(xW_I) \wole \rel(yW_J)
& \quad\iff\quad \rel(xW_I)^+ \supseteq \rel(yW_J)^+ \quad\text{and}\quad \rel(xW_I)^- \subseteq \rel(yW_J)^- \\
& \quad\iff\quad \overline{\rel}(xW_I)^+ \subseteq \overline{\rel}(yW_J)^+ \quad\text{and}\quad \overline{\rel}(xW_I)^- \supseteq \overline{\rel}(yW_J)^- \\
& \quad\iff\quad xW_I \wole yW_J.
\qedhere
\end{align*}
\end{proof}

\begin{corollary}
\label{coro:latticeWOFP}
The weak order on~$\WOFP(\Phi)$ is a lattice with meet and join
\[
\rel(xW_I) \meetWOFP \rel(yW_J) = \rel(xW_I \meetFWO yW_J)
\quad\text{and}\quad
\rel(xW_I) \joinWOFP \rel(yW_J) = \rel(xW_I \joinFWO yW_J).
\]
\end{corollary}

\begin{remark}
\label{rem:notSublatticeWOFP}
To conclude, note that the weak order on~$\WOFP(\Phi)$ is a lattice but not a sublattice of the weak order on~$\IPos(\Phi)$, nor on~$\WOIP(\Phi)$.
This was already observed in~\cite[Rem.~31]{ChatelPilaudPons} in type~$A$.
For example, in type~$A_2$ we have
\[
\{-\alpha_1, \alpha_2\} \joinC \varnothing = \{-\alpha_1, \alpha_2\} \joinWOIP \varnothing = \{\alpha_2\}
\quad\text{while}\quad
\{-\alpha_1, \alpha_2\} \joinWOIP \varnothing = \{\alpha_2, \alpha_1+\alpha_2\}.
\]
\end{remark}

\subsection{Generalized associahedra}
\label{subsec:associahedra}

We now consider $\Phi$-posets corresponding to the vertices, the intervals and the faces of the generalized associahedra of type~$\Phi$.
These polytopes provide geometric realizations of the type~$\Phi$ cluster complex, in connection to the type~$\Phi$ cluster algebra of S.~Fomin and A.~Zelevinsky~\cite{FominZelevinsky-ClusterAlgebrasI, FominZelevinsky-ClusterAlgebrasII}.
A first realization was constructed by F.~Chapoton, S.~Fomin and A.~Zelevinsky in~\cite{ChapotonFominZelevinsky} based on the compatibility fan of~\cite{FominZelevinsky-YSystems, FominZelevinsky-ClusterAlgebrasII}. 
An alternative realization was constructed later by C.~Hohlweg, C.~Lange and H.~Thomas in~\cite{HohlwegLangeThomas} based on the Cambrian fan of N.~Reading and D.~Speyer~\cite{ReadingSpeyer}.

Although the sets of roots that we consider in this section have a strong connection to these geometric realizations (see Remarks~\ref{rem:verticesAssociahedron} and~\ref{rem:facesAssociahedron}), we do not really need for our purposes the precise definition of the geometry of these associahedra or of these Cambrian fans.
We rather need a combinatorial description of their vertices and faces.
The combinatorial model behind these constructions is the Cambrian lattice on sortable elements as developed by N.~Reading~\cite{Reading-cambrianLattices, Reading-CoxeterSortable, Reading-sortableElements}, which we briefly recall now.

Let~$c$ be a Coxeter element, \ie the product of the simple reflections of~$W$ in an arbitrary order.
The \defn{$c$-sorting word} of an element~$w \in W$ is the lexicographically smallest reduced expression for~$w$ in the word~$c^\infty \eqdef ccccc\cdots$.
We write this word as~$w = c_{I_1} \dots c_{I_k}$ where~$c_I$ is the subword of~$c$ consisting only of the simple reflections in~$I$.
An element~$w \in W$ is \defn{$c$-sortable} when these subsets are nested:~$I_1 \supseteq I_2 \supseteq \cdots \supseteq I_k$.
An element~$w \in W$ is \defn{$c$-antisortable} when~$w w_\circ$ is $(c^{-1})$-sortable.
See~\cite{Reading-CoxeterSortable} for details on Coxeter sortable elements and their connections to other Coxeter-Catalan families.

For an element~$w \in W$, we denote by~$\projDown^c(w)$ the maximal $c$-sortable element below~$w$ in weak order and by~$\projUp_c(w)$ the minimal $c$-antisortable element above~$w$ in weak order.
The projection maps~$\projDown^c$ and~$\projUp_c$ can also be defined inductively, see~\cite{Reading-sortableElements}.
Here, we only need that these maps are order preserving projections from~$W$ to sortable (resp.~antisortable) elements, and that their fibers are intervals of the weak order of the form~$[\projDown^c(w), \projUp_c(w)]$.
Therefore, they define a lattice congruence~$\equiv_c$ of the weak order, called the \defn{$c$-Cambrian congruence}.
The quotient of the weak order by this congruence~$\equiv_c$ is the \defn{$c$-Cambrian lattice}.
It is isomorphic to the sublattice of the weak order induced by $c$-sortable (or $c$-antisortable) elements.
In particular, for two $c$-Cambrian classes~$X,Y$, we have~$X \wole Y$ in the $c$-Cambrian lattice $\iff$ there exists~$x \in X$ and~$y \in Y$ such that~$x \wole y$ in the weak order on~$W$ $\iff$ $\projDown^c(X) \wole \projDown^c(Y)$ $\iff$ $\projUp_c(X) \wole \projUp_c(Y)$.
We denote by~$X \meet_c Y$ and~$X \join_c Y$ the meet and join of the two $c$-Cambrian classes~$X,Y$.

Let~$w_\circ(c) = q_1 \dots q_N$ denote the $c$-sorting word for the longest element~$w_\circ$.
It defines an order on~$\Phi^+$ by~$\alpha_{q_1} <_c q_1 (\alpha_{q_2}) <_c q_1 q_2 (\alpha_{q_3}) <_c \dots <_c q_1 \dots q_{N-1} (\alpha_{q_N})$.
A subset~$\rel$ of positive roots is called \defn{$c$-aligned} if for any~$\alpha <_c \beta$ such that~$\alpha + \beta \in {\rel}$, we have~$\alpha \in {\rel}$.
It is known that~$w \in W$ is $c$-sortable if and only if its inversion set~$\inv(w)$ is $c$-aligned~\cite{Reading-sortableElements}.

\subsubsection{Elements}
\label{subsubsec:COEP}

For a $c$-Cambrian class~$X$, we consider the $\Phi$-poset
\[
\rel(X) \eqdef \bigcap_{w \in X} \rel(w) = \rel\big(\projDown^c(X)\big) \cap \rel\big(\projUp_c(X)\big) = \rel\big(\projDown^c(X)\big)^- \sqcup \rel\big(\projUp_c(X)\big)^+.
\]
Note that by definition, $\rel(X)$ coincides with the weak order interval poset~$\rel \big( \projDown^c(X), \projUp_c(X) \big)$.
We say that~$\rel(X)$ is a \defn{$c$-Cambrian order element poset} and we denote the collection of all such $\Phi$-posets by~$\COEP(c) \eqdef \set{\rel(X)}{X \text{ $c$-Cambrian class}}$.

\begin{remark}
Table~\ref{table:numerology} reports the cardinality of~$\COEP(c)$ in type~$A_n$, $B_n$, $C_n$ and~$D_n$ for small values of~$n$.
Observe that this cardinality is independent of the choice of the Coxeter element~$c$, and is the Coxeter-Catalan number (counting many related objects from clusters of type~$\Phi$ to non-crossing partitions of~$W$):
\[
|\COEP(c)| = \mathrm{Cat}(W) = \prod_{i \in [n]} \frac{1+d_i}{d_i},
\]
where~$(d_1, \dots, d_n)$ still denote the degrees of~$W$.
\end{remark}

\begin{remark}
\label{rem:verticesAssociahedron}
Geometrically, $\rel(X)$ is the set of roots of~$\Phi$ not contained in the cone of the vertex corresponding to~$X$ in the generalized associahedron~$\Asso(c)$ of C.~Hohlweg, C.~Lange and H.~Thomas in~\cite{HohlwegLangeThomas}.
See \fref{fig:facesAssociahedron}.
\end{remark}

Let us now take a little detour to comment on a conjectured characterization of these $\Phi$-posets, inspired from a similar characterization in type~$A$ proved in~\cite[Prop.~60]{ChatelPilaudPons}.
Note that it uses the $c$-Cambrian order interval posets formally defined in the next section and characterized in Proposition~\ref{prop:characterizationCOIP}.
It also requires the notion of $c$-snakes.
A \defn{$c$-snake} in a $\Phi$-poset~$\rel$ is a sequence of roots~$\alpha_1, \dots, \alpha_p \in \rel$ such that
\begin{itemize}
\item either~$\alpha_{2i} \in \Phi^-$, $\alpha_{2i+1} \in \Phi^+$ and $\alpha_1 <_c -\alpha_2 >_c \alpha_3 <_c -\alpha_4 >_c \dots$
\item or~$\alpha_{2i} \in \Phi^+$, $\alpha_{2i+1} \in \Phi^-$ and $-\alpha_1 >_c \alpha_2 <_c -\alpha_3 >_c \alpha_4 <_c \dots$
\end{itemize}
A \defn{$c$-snake decomposition} of a root~$\alpha$ in~$\rel$ is a decomposition~$\alpha = \sum_{i \in [p]} \lambda_i \alpha_i$, where~$\lambda_i \in \N$ and~$\alpha_1, \dots, \alpha_p$ is a $c$-snake of~$\rel$.
The following conjectural characterization of $c$-Cambrian order element posets was proved in type~$A$ in~\cite[Prop.~60]{ChatelPilaudPons} and has been checked computationally for small Coxeter types using~\cite{Sage}.

\begin{conjecture}
\label{conj:characterizationCOEP}
A $\Phi$-poset~$\rel \in \IPos(\Phi)$ is in~$\COEP(c)$ if and only if it is in~$\COIP(c)$ (characterized in Proposition~\ref{prop:characterizationCOIP}) and any root~$\alpha \in \Phi$ admits a $c$-snake decomposition in~$\rel$.
\end{conjecture}

Even without this characterization, we can at least describe the weak order on these posets.

\begin{proposition}
\label{prop:weakOrderCOEP}
For any two $c$-Cambrian classes~$X$ and~$Y$, we have~$\rel(X) \wole \rel(Y)$ in the weak order on~$\COEP(c)$ if and only if~$X \wole Y$ in the $c$-Cambrian lattice.
\end{proposition}

\begin{proof}
By definition, a $c$-Cambrian class~$X$ admits both a minimal element~$\projDown^c(X)$ and a maximal element~$\projUp_c(X)$.
Therefore,~$\rel(X) = \rel \big( \projDown^c(X), \projUp_c(X) \big) \in \WOIP(\Phi)$.
Moreover, for two $c$-Cambrian classes~$X,Y$, Proposition~\ref{prop:weakOrderWOIP} implies that~$\rel(X) \wole \rel(Y)$ in the weak order on~$\WOIP(\Phi)$ if and only if~$\projDown^c(X) \wole \projDown^c(Y)$ and~$\projUp_c(X) \wole \projUp_c(Y)$ in weak order on~$W$.
But this is equivalent to~$X \wole Y$ in the $c$-Cambrian lattice as mentioned above.
\end{proof}

\begin{remark}
\label{rem:weakOrderCOEP}
In fact, $\rel(X) \wole \rel(Y) \iff \rel(X)^+ \supseteq \rel(Y)^+ \iff \rel(X)^- \subseteq \rel(Y)^- \iff X \wole Y$.
\end{remark}

\begin{corollary}
\label{coro:latticeCOEP}
For any Coxeter element~$c$, the weak order on~$\COEP(c)$ is a lattice with meet~and~join
\[
\rel(X) \meetCOEP \rel(Y) = \rel(X \meet_c Y)
\qquad\text{and}\qquad
{\rel}(X) \joinCOEP \rel(Y) = \rel(X \join_c Y).
\]
\end{corollary}

Although it anticipates on the $c$-Cambrian order interval posets studied in the next section, let us state the following result that will be a direct consequence of Corollary~\ref{coro:latticeCOIP} and Proposition~\ref{prop:sublatticeCOIP}.

\begin{proposition}
For any Coxeter element~$c$, the set~$\COEP(c)$ induces a sublattice of the weak order on~$\COIP(c)$ and thus also a sublattice of the weak order on~$\WOIP(\Phi)$.
\end{proposition}

We conclude our discussion on~$\COEP(c)$ with one more conjecture, which was proved in type~$A$ in~\cite[Coro.~88]{ChatelPilaudPons} and checked computationally for small Coxeter types using~\cite{Sage}.
Note that there is little hope to attack this conjecture before proving either Conjecture~\ref{conj:characterizationCOEP} or Conjecture~\ref{conj:sublatticeCOIP}.

\begin{conjecture}
\label{conj:sublatticeCOEP}
For any Coxeter element~$c$, the set~$\COEP(c)$ induces a sublattice of the weak order on~$\IPos(\Phi)$.
\end{conjecture}

\subsubsection{Intervals}

For two $c$-Cambrian classes~$X,X'$ with~$X \wole X'$ in the $c$-Cambrian order, we denote by~$[X,X'] \eqdef \set{Y \text{ $c$-Cambrian class}}{X \wole Y \wole X'}$ the \defn{$c$-Cambrian order interval} between~$X$ and~$X'$.
We associate to each $c$-Cambrian order interval~$[X,X']$ the $\Phi$-poset
\[
\rel(X,X') \eqdef \bigcap_{Y \in [X,X']} \rel(Y) = \rel(X) \cap \rel(X') = \rel(X)^- \cup \rel(X')^+.
\]
Note that by definition, $\rel(X,X')$ coincides with the weak order interval poset~$\rel \big( \projDown^c(X), \projUp_c(X') \big)$.
We say that~$\rel(X,X')$ is a \defn{$c$-Cambrian order interval poset} and we denote the collection of all such $\Phi$-posets by~$\COIP(c) \eqdef \set{\rel(X,X')}{X,X' \text{ $c$-Cambrian classes}, \, X \wole X'}$.

\begin{remark}
Table~\ref{table:numerology} reports the cardinality of~$\COIP(c)$ in type~$A_n$, $B_n$, $C_n$ and~$D_n$ for small values of~$n$ and different choices of the Coxeter element~$c$.
We have denoted by bip the bipartite Coxeter element, and by lin the linear one (with the special vertex first in type~$B/C$ and the two special vertices first in type~$D$).
Note that in contrast to $\COEP(c)$, the cardinality of~$\COIP(c)$ depends on the choice of the Coxeter element~$c$ (this comes from the fact that the $c$-Cambrian lattices for different choices of Coxeter element~$c$ are not isomorphic and have distinct intervals, although they have the same number of elements).
\end{remark}

We now characterize the $\Phi$-posets in~$\COIP(c)$.

\begin{proposition}
\label{prop:characterizationCOIP}
A $\Phi$-poset~$\rel \in \IPos(\Phi)$ is in~$\COIP(c)$ if and only if $\alpha + \beta \in \rel$ and~$\alpha <_c \beta$ implies $\beta \in \rel$ for all~$\alpha, \beta \in \Phi^+$ (resp.~$\alpha \in \rel$ for all~$\alpha, \beta \in \Phi^-$).
\end{proposition}

\begin{proof}
Consider a $\Phi$-poset~$\rel \in \IPos(\Phi)$.
By definition, $\rel$ is in~$\COIP(c)$ if and only if $\rel = \rel(w,w')$ is in $\WOIP(\Phi)$ where~$w$ is $c$-sortable while~$w'$ is $c$-antisortable.
However, $w$ is $c$-sortable if and only if~$\inv(w) = \Phi^+ \cap -{\rel(w)} = -{\rel(w)}^- = -{\rel(w,w')}^- = -\rel^-$ is $c$-aligned, \ie if and only if~$\alpha + \beta \in \rel^- \implies \alpha \in \rel^-$ for any~$\alpha <_c \beta$.
Similarly, $w'$ is $c$-antisortable if and only if~${\alpha + \beta \in \rel^+ \implies \beta \in \rel^+}$ for any~$\alpha <_c \beta$.
\end{proof}

\begin{proposition}
\label{prop:weakOrderCOIP}
For two $c$-Cambrian intervals~$X \wole X'$ and~$Y \wole Y'$, we have~${\rel(X,X') \wole \rel(Y,Y')}$ in the weak order on~$\COIP(c)$ if and only if~$X \wole Y$ and~$X' \wole Y'$ in the $c$-Cambrian order.
\end{proposition}

\begin{proof}
By definition of~$\rel(X,X')$ and Remark~\ref{rem:weakOrderCOEP}, we obtain
\[
\begin{array}[b]{c@{\quad\iff\quad}c@{\quad\text{and}\quad}c}
\rel(X,X') \wole \rel(Y,Y')
& \rel(X,X')^+ \supseteq \rel(Y,Y')^+ & \rel(X,X')^- \subseteq \rel(Y,Y')^- \\
& \rel(X')^+ \supseteq \rel(Y')^+ & \rel(X)^- \subseteq \rel(Y)^- \\
& X' \wole Y' & X \wole Y.
\end{array}
\qedhere
\]
\end{proof}

\begin{corollary}
\label{coro:latticeCOIP}
For any Coxeter element~$c$, the weak order on~$\COIP(c)$ is a lattice with meet~and~join
\[
\rel(X,X') \meetCOIP \rel(Y,Y') = \rel(X \meet_c Y, X' \meet_c Y')
\;\;\text{and}\;\;
\rel(X,X') \joinCOIP \rel(Y,Y') = \rel(X \join_c Y, X' \join_c Y').
\]
\end{corollary}

The following statement connects this lattice structure on~$\COIP(c)$ with that on~$\WOIP(\Phi)$.

\begin{proposition}
\label{prop:sublatticeCOIP}
For any Coxeter element~$c$, the set~$\COIP(c)$ induces a sublattice of the weak order on~$\WOIP(\Phi)$.
\end{proposition}

\begin{proof}
Consider two $c$-Cambrian intervals~$X \wole X'$ and~$Y \wole Y'$.
By Corollary~\ref{coro:latticeWOIP}, we have
\begin{align*}
\rel(X,X') \meetWOIP \rel(Y,Y') 
& = \rel \big( \projDown^c(X), \projUp_c(X') \big) \meetWOIP \rel \big( \projDown^c(Y), \projUp_c(Y') \big) \\
& = \rel \big( \projDown^c(X) \meetWO \projDown^c(Y), \projDown^c(X') \meetWO \projDown^c(Y') \big) \\
& = \rel \big( \projDown^c(X \meet_c Y), \projDown^c(X' \meet_c Y') \big),
\end{align*}
where the last equality follows from the fact that~$c$-sortable elements (resp.~$c$-antisortable elements) induce a sublattice of the weak order.
\end{proof}

The following conjecture indicates that~$\COIP(c)$ behaves much better than~$\WOIP(\Phi)$ as subposet of~$\IPos(\Phi)$.
This conjecture unfortunately remains open for now but was proved in type~$A$ in~\cite[Coro.~82]{ChatelPilaudPons} and verified for small Coxeter types using~\cite{Sage}.
Note that it is not implied by Proposition~\ref{prop:sublatticeCOIP} since~$\WOIP(\Phi)$ is not a sublattice of~$\IPos(\Phi)$.
Observe also that it would imply Conjecture~\ref{conj:sublatticeCOEP}.

\begin{conjecture}
\label{conj:sublatticeCOIP}
For any Coxeter element~$c$, the set~$\COIP(c)$ induces a sublattice of the weak order on~$\IPos(\Phi)$.
\end{conjecture}

\subsubsection{Faces}

To remain at a combinatorial level and avoid geometric descriptions (see Remark~\ref{rem:facesAssociahedron}), we consider a combinatorial model for the faces of the associahedron~$\Asso(c)$ that rely on results of~\cite[Sec.~4]{DermenjianHohlwegPilaud}.
The $c$-Cambrian congruence~$\equiv_c$ extends to the \defn{$c$-Cambrian facial congruence} on all faces of the permutahedron~$\Perm(W)$ defined by~${xW_I \equiv_c yW_J \iff x \equiv_c y \text{ and } xw_{\circ,I} \equiv_c yw_{\circ,J}}$.
This relation is a lattice congruence of the facial weak order on faces of the permutahedron~$\Perm(W)$ \cite[Prop.~4.12]{DermenjianHohlwegPilaud} and we denote by~$\ProjDown^c$ and~$\ProjUp_c$ its down and up projections.
Moreover, the $c$-Cambrian facial congruence classes precisely correspond to the faces of the associahedron~$\Asso(c)$ of~\cite{HohlwegLangeThomas}.

For a $c$-Cambrian facial congruence class~$F$, we consider the $\Phi$-poset
\[
\rel(F) \eqdef \bigcap_{xW_I \in F} \rel(xW_I) = \rel \big( \ProjDown^c(F) \big)^- \cap \rel \big( \ProjUp_c(F) \big)^+.
\]
Note that if~$\ProjDown^c(F) = xW_I$ and~$\ProjUp_c(F) = yW_J$, then $\rel(F)$ coincides with the weak order interval poset~$\rel(x, yw_{\circ,J})$.
We say that~$\rel(F)$ is a \defn{$c$-Cambrian order face poset} and denote the set of such $\Phi$-posets by~$\COFP(c) \eqdef \set{\rel(F)}{F \text{ $c$-Cambrian facial congruence class}}$.

\begin{remark}
Table~\ref{table:numerology} reports the cardinality of~$\COFP(c)$ in type~$A_n$, $B_n$, $C_n$ and~$D_n$ for small values of~$n$.
Note that this cardinality is again independent of the choice of the Coxeter element~$c$ (it is the number of faces in the generalized associahedron, \ie the number of partial clusters in the corresponding cluster algebra).
\end{remark}

\begin{remark}
\label{rem:facesAssociahedron}
Geometrically, $\rel(F)$ is the set of roots of~$\Phi$ not contained in the cone of the face~$F$ in the generalized associahedron~$\Asso(c)$ of C.~Hohlweg, C.~Lange and H.~Thomas in~\cite{HohlwegLangeThomas}.
See \fref{fig:facesAssociahedron}.

\begin{figure}[t]
\DeclareDocumentCommand{\rs}{ O{1.1cm} O{->} m m O{0}} {
	\def \radius {#1}
	\def \inputPoints{#3}
	\def \includeRoots{#4}
	\def \style {#2}
	\def \initialRotation {#5}
	
	\pgfmathtruncatemacro{\points}{\inputPoints * 2}
	\pgfmathsetmacro{\degrees}{360 / \points}
	
	\coordinate (0) at (0,0);
	
	\foreach \x in {1,...,\points}{%
		\pgfmathsetmacro{\location}{(\points+(\x-1))*\degrees + \initialRotation}
		\ifthenelse{\equal{\inputPoints}{3}}{
			\coordinate (\x) at (\location:\radius);
		}{
			\coordinate (\x) at (\location:1.5*\radius*(1 + 0.41*Mod(\x,2);));		
		}
	}

	\ifthenelse{\equal{\includeRoots}{}}{
	}{
		\foreach \x in {1,...,\points}{%
			\edef \showPoint {0};

			\foreach \y in \includeRoots {
				\ifthenelse{\equal{\x}{\y}}{
					\xdef \showPoint {1};
				}{}
			}
			
			\ifthenelse{\equal{\showPoint}{1}}{
				\draw[->, ultra thick] (0) -- (\x);
			}{}
		}
	}  
}

\centerline{
	\begin{tikzpicture}
		[scale=2,
		bface/.style={color=blue},
		rface/.style={color=red},
		gface/.style={color=black}
		]
		%
		\coordinate (e) at (0,0.42);
		\coordinate (s) at (-1,1);
		\coordinate (t/ts) at (2,1.5);
		\coordinate (st) at (-1,2);
		\coordinate (sts) at (0,2.58);
		%
		\coordinate (Ws) at (-0.47,0.69);
		\coordinate (Wt) at (1,.96);
		\coordinate (sWt) at (-1,1.5);
		\coordinate (stWs) at (-0.47,2.31);
		\coordinate (tsWt) at (1,2.04);
		%
		\coordinate (W) at (0,1.5);
		%
		\draw (e) -- (s);
		\draw (e) -- (t/ts);
		\draw (s) -- (st);
		\draw (st) -- (sts);
		\draw (t/ts) -- (sts);
		%
		\begin{scope}[shift={(e)}, scale=0.25, bface]
			\begin{scope}[scale=0.8]
				\rs{3}{1,2,3}[210]
			\end{scope}
		\end{scope}
		\begin{scope}[shift={(s)}, scale=0.25, bface]
			\begin{scope}[scale=0.8]
				\rs{3}{1,2,6}[210]
			\end{scope}
		\end{scope}
		\begin{scope}[shift={(st)}, scale=0.25, bface]
			\begin{scope}[scale=0.8]
				\rs{3}{1,5,6}[210]
			\end{scope}
		\end{scope}
		\begin{scope}[shift={(t/ts)}, scale=0.25, bface]
			\begin{scope}[scale=0.8]
				\rs{3}{3,4}[210]
			\end{scope}
		\end{scope}
		\begin{scope}[shift={(sts)}, scale=0.25, bface]
			\begin{scope}[scale=0.8]
				\rs{3}{4,5,6}[210]
			\end{scope}
		\end{scope}
		%
		\begin{scope}[shift={(Ws)}, scale=0.25,rface]
			\begin{scope}[scale=0.8]
				\rs{3}{1,2}[210]
			\end{scope}
		\end{scope}
		\begin{scope}[shift={(Wt)}, scale=0.25, rface]
			\begin{scope}[scale=0.8]
				\rs{3}{2,3}[210]
			\end{scope}
		\end{scope}
		\begin{scope}[shift={(sWt)}, scale=0.25, rface]
			\begin{scope}[scale=0.8]
				\rs{3}{1,6}[210]
			\end{scope}
		\end{scope}
		\begin{scope}[shift={(stWs)}, scale=0.25, rface]
			\begin{scope}[scale=0.8]
				\rs{3}{5,6}[210]
			\end{scope}
		\end{scope}
		\begin{scope}[shift={(tsWt)}, scale=0.25, rface]
			\begin{scope}[scale=0.8]
				\rs{3}{4,5}[210]
			\end{scope}
		\end{scope}
	\end{tikzpicture}
	\hspace{.5cm}
	\begin{tikzpicture}
		[scale=1.1,
		bface/.style={color=blue},
		rface/.style={color=red},
		gface/.style={color=black}
		]
		%
		\coordinate (e) at (0,0);
		\coordinate (s) at (-1.4,0.6);
		\coordinate (t/ts/tst) at (4,2);
		\coordinate (st) at (-2,2);
		\coordinate (sts) at (-1.4,3.4);
		\coordinate (stst) at (0,4);
		%
		\coordinate (Ws) at (-0.7,0.3);
		\coordinate (Wt) at (2,1);
		\coordinate (sWt) at (-1.7,1.3);
		\coordinate (stWs) at (-1.7,2.7);
		\coordinate (tstWs) at (2,3);
		\coordinate (stsWt) at (-0.7,3.7);
		%
		\coordinate (W) at (0,2);
		%
		\draw (e) -- (s);
		\draw (e) -- (t/ts/tst);
		\draw (s) -- (st);
		\draw (st) -- (sts);
		\draw (t/ts/tst) -- (stst);
		\draw (sts) -- (stst);
		%
		\begin{scope}[shift={(e)}, scale=0.25, bface]
			\begin{scope}[scale=0.8]
				\rs{4}{1,2,3,4}[202]
			\end{scope}
		\end{scope}
		\begin{scope}[shift={(s)}, scale=0.25, bface]
			\begin{scope}[scale=0.8]
				\rs{4}{1,2,3,8}[202]
			\end{scope}
		\end{scope}
		\begin{scope}[shift={(st)}, scale=0.25, bface]
			\begin{scope}[scale=0.8]
				\rs{4}{1,2,7,8}[202]
			\end{scope}
		\end{scope}
		\begin{scope}[shift={(sts)}, scale=0.25, bface]
			\begin{scope}[scale=0.8]
				\rs{4}{1,6,7,8}[202]
			\end{scope}
		\end{scope}
		\begin{scope}[shift={(stst)}, scale=0.25, bface]
			\begin{scope}[scale=0.8]
				\rs{4}{5,6,7,8}[202]
			\end{scope}
		\end{scope}
		\begin{scope}[shift={(t/ts/tst)}, scale=0.25, bface]
			\begin{scope}[scale=0.8]
				\rs{4}{4,5}[202]
			\end{scope}
		\end{scope}
		%
		\begin{scope}[shift={(Ws)}, scale=0.25, rface]
			\begin{scope}[scale=0.8]
				\rs{4}{1,2,3}[202]
			\end{scope}
		\end{scope}
		\begin{scope}[shift={(sWt)}, scale=0.25, rface]
			\begin{scope}[scale=0.8]
				\rs{4}{1,2,8}[202]
			\end{scope}
		\end{scope}
		\begin{scope}[shift={(stWs)}, scale=0.25, rface]
			\begin{scope}[scale=0.8]
				\rs{4}{1,7,8}[202]
			\end{scope}
		\end{scope}
		\begin{scope}[shift={(stsWt)}, scale=0.25, rface]
			\begin{scope}[scale=0.8]
				\rs{4}{6,7,8}[202]
			\end{scope}
		\end{scope}
		\begin{scope}[shift={(tstWs)}, scale=0.25, rface]
			\begin{scope}[scale=0.8]
				\rs{4}{5,6,7}[202]
			\end{scope}
		\end{scope}
		\begin{scope}[shift={(Wt)}, scale=0.25, rface]
			\begin{scope}[scale=0.8]
				\rs{4}{2,3,4}[202]
			\end{scope}
		\end{scope}
	\end{tikzpicture}
}
	\caption{The sets~$\rel(F)$ for the faces~$F$ of the $c$-associahedron in type~$A_2$~(left) and~$B_2$ (right). Note that positive roots point downwards.}
	\label{fig:facesAssociahedron}
\end{figure}
\end{remark}

It would be interesting to have a characterization of the $\Phi$-posets in~$\COFP(c)$ similar to that given in~\cite{ChatelPilaudPons} in type~$A$ (see \cite[Prop.~46]{ChatelPilaudPons} for the Tamari faces and \cite[Prop.~63]{ChatelPilaudPons} for the type~$A$ Cambrian faces in general).

Here, we just connect the weak order on~$\COFP(c)$ with the facial weak order on the associahedron~$\Asso(c)$ considered in~\cite[Sec.~4.7.2]{DermenjianHohlwegPilaud}.
This order is the quotient of the facial weak order on the faces of the permutahedron~$\Perm(W)$ by the $c$-Cambrian facial congruence~$\equiv_c$.

\begin{proposition}
\label{prop:weakOrderCOFP}
For any two $c$-Cambrian facial congruence classes~$F$ and~$G$, we have~${\rel(F) \wole \rel(G)}$ in the weak order on~$\COFP(c)$ if and only if~$F \wole G$ in the $c$-Cambrian facial lattice.
\end{proposition}

\begin{proof}
This is immediate from the definitions:
\[
\begin{array}[b]{c@{\quad\iff\quad}c@{\quad\text{and}\quad}c}
\rel(F) \wole \rel(G)
& \rel \big( \ProjUp_c(F) \big)^+ \supseteq \rel \big( \ProjUp_c(G) \big)^+ & \rel \big( \ProjDown^c(F) \big)^- \subseteq \rel \big( \ProjDown^c(G) \big)^- \\
& \ProjUp_c(F) \wole \ProjUp_c(G) & \ProjDown^c(F) \wole \ProjDown^c(G) \\
& \multicolumn{2}{@{}l}{F \wole G.}
\end{array}
\qedhere
\]
\end{proof}

\begin{corollary}
\label{coro:latticeCOFP}
For any Coxeter element~$c$, the weak order on~$\COFP(c)$ is a lattice.
\end{corollary}

\begin{remark}
To conclude, note that the weak order on~$\COFP(c)$ is a lattice but not a sublattice of the weak order on~$\IPos(\Phi)$, nor on~$\WOIP(\Phi)$, nor on~$\COIP(c)$.
This was already observed in~\cite[Rem.~47]{ChatelPilaudPons} in type~$A$.
For example, consider the example of Remark~\ref{rem:notSublatticeWOFP} for the Coxeter element~$s_1s_2$ in type~$A_2$.
\end{remark}

\subsection{Cube}
\label{subsec:cube}

To conclude this paper, we consider $\Phi$-posets corresponding to the vertices, the intervals and the faces of the cube (see Remarks~\ref{rem:verticesCube} and~\ref{rem:facesCube}), corresponding to the descent congruence on~$W$.
Recall that a (left) \defn{descent} of $w \in W$ is a simple root~$\alpha \in \Delta$ such that~$s_\alpha w \wole w$, or equivalently~$\alpha \in \inv(w)$.
The \defn{descent set} of~$w$ is~$\des(w) \eqdef \inv(w) \cap \Delta$.
The \defn{descent class} of~$w$ is the set of elements of~$W$ that have precisely the same descent set as~$w$.
Note that descent classes correspond to subsets of~$\Delta$: for~$A \subseteq \Delta$, we denote by~$Z_A$ the descent class of elements of~$W$ with $A$ as descent set.
These classes define the \defn{descent congruence} on~$W$, whose down and up projections we denote by~$\projDown^d$ and~$\projUp_d$.

\subsubsection{Elements}
\label{subsubsec:BOEP}

For a subset~$A \subseteq \Delta$ corresponding to the descent class~$Z_A$, we consider the $\Phi$-poset
\begin{align*}
\rel(A) & \eqdef \cl{\big( -A \sqcup (\Delta \ssm A) \big)} = \Phi \cap \N \big( -A \sqcup (\Delta \ssm A) \big) \\ 
& = \!\! \bigcap_{w \in Z_A} \!\! \rel(w) = \rel\big(\projDown^d(Z_A)\big) \cap \rel\big(\projUp_d(Z_A)\big) = \rel\big(\projDown^d(Z_A)\big)^- \sqcup \rel\big(\projUp_d(Z_A)\big)^+ \!\!.
\end{align*}
Note that by definition, $\rel(A)$ coincides with the weak order interval poset~$\rel \big( \projDown^c(Z_A), \projUp_c(Z_A) \big)$.
We say that~$\rel(A)$ is a \defn{boolean order element poset} and we denote the collection of all such $\Phi$-posets by~$\BOEP(\Phi) \eqdef \set{\rel(A)}{A \subseteq \Delta}$.
Note that there are $2^n$ many $\Phi$-posets in~$\BOEP(\Phi)$, see Table~\ref{table:numerology}.

\begin{remark}
\label{rem:verticesCube}
Geometrically, $\rel(A)$ is the set of roots of~$\Phi$ not contained in the cone of the vertex corresponding to~$A$ in the parallelepiped generated by the simple roots~$\Delta$.
See \fref{fig:facesCube}.
\end{remark}

These $\Phi$-posets are characterized in the next statement.
Its proof is delayed to Section~\ref{subsubsec:BOIP} as it requires the characterization of the boolean order interval posets.

\begin{proposition}
\label{prop:characterizationBOEP}
A $\Phi$-poset~$\rel \in \IPos(\Phi)$ is in~$\BOEP(\Phi)$ if and only if
\begin{enumerate}[(i)]
\item $\alpha + \beta \in \rel \implies \alpha \in \rel \text{ and } \beta \in \rel$ for all~$\alpha, \beta \in \Phi^+$ and all~$\alpha, \beta \in \Phi^-$,
\item $\alpha \in \rel$ or~$-\alpha \in \rel$ for any simple root~$\alpha \in \Delta$.
\end{enumerate}
\end{proposition}

The following statement characterizes the weak order induced by~$\BOEP(\Phi)$.

\begin{proposition}
\label{prop:weakOrderBOEP}
For any subsets~$A, B \subseteq \Delta$, we have~${\rel(A) \wole \rel(B)}$ in the weak order on~$\BOEP(\Phi)$ if and only if~$A \subseteq B$ in boolean order.
\end{proposition}

\begin{proof}
From the definition~$\rel(A) = \Phi \cap \N \big( -A \sqcup (\Delta \ssm A) \big)$, we obtain that
\[
\begin{array}[b]{c@{\quad\iff\quad}c@{\quad\text{and}\quad}c}
\rel(A) \wole \rel(B)
& \rel(A)^+ \supseteq \rel(B)^+ & \rel(A)^- \subseteq \rel(B)^- \\
& \Delta \ssm A \supseteq \Delta \ssm B & A \subseteq B.
\end{array}
\qedhere
\]
\end{proof}

\begin{remark}
\label{rem:weakOrderBOEP}
In fact, $\rel(A) \wole \rel(B) \iff \rel(A)^+ \supseteq \rel(B)^+ \iff \rel(A)^- \subseteq \rel(B)^- \iff A \subseteq B$.
\end{remark}

\begin{corollary}
\label{coro:latticeBOEP}
The weak order on~$\BOEP(\Phi)$ is a lattice with meet and join
\[
\rel(A) \meetBOEP \rel(B) = \rel(A \cap B)
\qquad\text{and}\qquad
\rel(A) \joinBOEP \rel(B) = \rel(A \cup B).
\]
\end{corollary}

Although it anticipates on the boolean order interval posets studied in the next section, let us state the following result that will be a direct consequence of Corollary~\ref{coro:latticeBOIP} and Proposition~\ref{prop:sublatticeBOIP}.

\begin{proposition}
\label{prop:sublatticeBOEP}
The set~$\BOEP(\Phi)$ induces a sublattice of the weak order on~$\BOIP(\Phi)$ and therefore on the weak orders on~$\IPos(\Phi)$, on~$\WOIP(\Phi)$ and on~$\COIP(c)$ for all Coxeter element~$c$.
\end{proposition}

\subsubsection{Intervals and Faces}
\label{subsubsec:BOIP}

We finally consider intervals in the boolean order, or equivalently faces of the cube (see Remark~\ref{rem:facesCube}).
For two subsets~$A \subseteq A'$ of~$\Delta$, we consider
\[
\rel(A,A') \eqdef = \bigcap_{A \subseteq B \subseteq A'} \rel(B) = \rel(A) \cap \rel(A') = \rel(A)^- \sqcup \rel(A')^+.
\]
Note that by definition, $\rel(A,A')$ coincides with the weak order interval poset~$\rel \big( \projDown^c(Z_A), \projUp_c(Z_{A'}) \big)$.
Observe also that $\BOIP(\Phi) \subseteq \COIP(c)$ for any Coxeter element~$c$ since the descent congruence coarsens the $c$-Cambrian congruence.
We say that $\rel(A,A')$ is a \defn{boolean order interval poset} and we denote the set of such $\Phi$-posets by~$\BOIP(\Phi) \eqdef \set{\rel(A,A')}{A \subseteq A' \subseteq \Delta}$.

\begin{remark}
\label{rem:facesCube}
Geometrically, $\rel(A,A')$ is the set of roots of~$\Phi$ not contained in the cone of the face corresponding to~$A \subseteq A'$ in the parallelepiped generated by the simple roots~$\Delta$.
See \fref{fig:facesCube}.

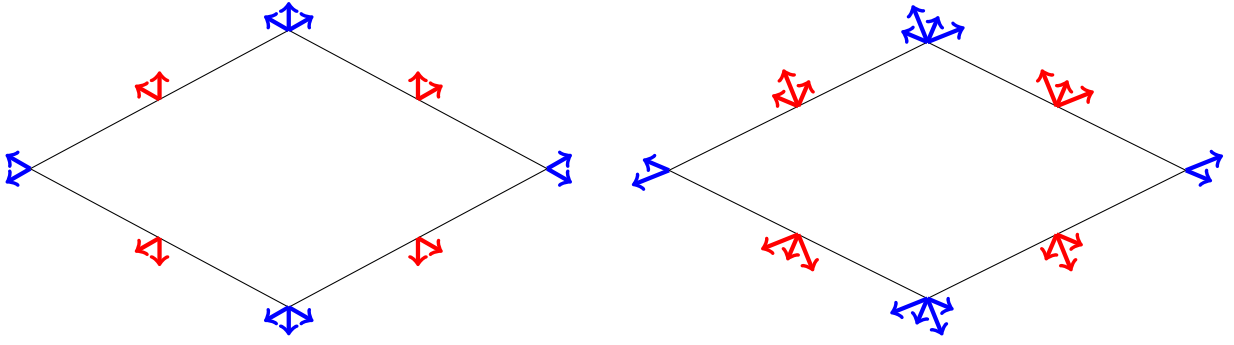
\begin{figure}[t]
\DeclareDocumentCommand{\rs}{ O{1.1cm} O{->} m m O{0}} {
	\def \radius {#1}
	\def \inputPoints{#3}
	\def \includeRoots{#4}
	\def \style {#2}
	\def \initialRotation {#5}
	
	\pgfmathtruncatemacro{\points}{\inputPoints * 2}
	\pgfmathsetmacro{\degrees}{360 / \points}
	
	\coordinate (0) at (0,0);
	
	\foreach \x in {1,...,\points}{%
		\pgfmathsetmacro{\location}{(\points+(\x-1))*\degrees + \initialRotation}
		\ifthenelse{\equal{\inputPoints}{3}}{
			\coordinate (\x) at (\location:\radius);
		}{
			\coordinate (\x) at (\location:2*\radius*(1 + 0.41*Mod(\x,2);));		
		}
	}

	\ifthenelse{\equal{\includeRoots}{}}{
	}{
		\foreach \x in {1,...,\points}{%
			\edef \showPoint {0};

			\foreach \y in \includeRoots {
				\ifthenelse{\equal{\x}{\y}}{
					\xdef \showPoint {1};
				}{}
			}
			
			\ifthenelse{\equal{\showPoint}{1}}{
				\draw[->, ultra thick] (0) -- (\x);
			}{}
		}
	}  
}

\centerline{
	\begin{tikzpicture}
		[scale=1.7,
		bface/.style={color=blue},
		rface/.style={color=red},
		gface/.style={color=black}
		]
		%
		\coordinate (e) at (0,0.42);
		\coordinate (s/st) at (-2,1.5);
		\coordinate (t/ts) at (2,1.5);
		\coordinate (sts) at (0,2.58);
		%
		\coordinate (Ws) at (-1,0.96);
		\coordinate (Wt) at (1,.96);
		\coordinate (stWs) at (-1,2.04);
		\coordinate (tsWt) at (1,2.04);
		%
		\coordinate (W) at (0,1.5);
		%
		\draw (e) -- (s/st);
		\draw (e) -- (t/ts);
		\draw (s/st) -- (sts);
		\draw (t/ts) -- (sts);
		%
		\begin{scope}[shift={(e)}, scale=0.25, bface]
			\begin{scope}[scale=0.8]
				\rs{3}{1,2,3}[210]
			\end{scope}
		\end{scope}
		\begin{scope}[shift={(s/st)}, scale=0.25, bface]
			\begin{scope}[scale=0.8]
				\rs{3}{1,6}[210]
			\end{scope}
		\end{scope}
		\begin{scope}[shift={(t/ts)}, scale=0.25, bface]
			\begin{scope}[scale=0.8]
				\rs{3}{3,4}[210]
			\end{scope}
		\end{scope}
		\begin{scope}[shift={(sts)}, scale=0.25, bface]
			\begin{scope}[scale=0.8]
				\rs{3}{4,5,6}[210]
			\end{scope}
		\end{scope}
		%
		\begin{scope}[shift={(Ws)}, scale=0.25,rface]
			\begin{scope}[scale=0.8]
				\rs{3}{1,2}[210]
			\end{scope}
		\end{scope}
		\begin{scope}[shift={(Wt)}, scale=0.25, rface]
			\begin{scope}[scale=0.8]
				\rs{3}{2,3}[210]
			\end{scope}
		\end{scope}
		\begin{scope}[shift={(stWs)}, scale=0.25, rface]
			\begin{scope}[scale=0.8]
				\rs{3}{5,6}[210]
			\end{scope}
		\end{scope}
		\begin{scope}[shift={(tsWt)}, scale=0.25, rface]
			\begin{scope}[scale=0.8]
				\rs{3}{4,5}[210]
			\end{scope}
		\end{scope}
	\end{tikzpicture}
	\hspace{.5cm}
	\begin{tikzpicture}
		[scale=.85,
		bface/.style={color=blue},
		rface/.style={color=red},
		gface/.style={color=black}
		]
		%
		\coordinate (e) at (0,0);
		\coordinate (s/st/sts) at (-4,2);
		\coordinate (t/ts/tst) at (4,2);
		\coordinate (stst) at (0,4);
		%
		\coordinate (Ws) at (-2,1);
		\coordinate (Wt) at (2,1);
		\coordinate (tstWs) at (2,3);
		\coordinate (stsWt) at (-2,3);
		%
		\coordinate (W) at (0,2);
		%
		\draw (e) -- (s/st/sts);
		\draw (e) -- (t/ts/tst);
		\draw (t/ts/tst) -- (stst);
		\draw (s/st/sts) -- (stst);
		%
		\begin{scope}[shift={(e)}, scale=0.25, bface]
			\begin{scope}[scale=0.8]
				\rs{4}{1,2,3,4}[202]
			\end{scope}
		\end{scope}
		\begin{scope}[shift={(s/st/sts)}, scale=0.25, bface]
			\begin{scope}[scale=0.8]
				\rs{4}{1,8}[202]
			\end{scope}
		\end{scope}
		\begin{scope}[shift={(stst)}, scale=0.25, bface]
			\begin{scope}[scale=0.8]
				\rs{4}{5,6,7,8}[202]
			\end{scope}
		\end{scope}
		\begin{scope}[shift={(t/ts/tst)}, scale=0.25, bface]
			\begin{scope}[scale=0.8]
				\rs{4}{4,5}[202]
			\end{scope}
		\end{scope}
		%
		\begin{scope}[shift={(Ws)}, scale=0.25, rface]
			\begin{scope}[scale=0.8]
				\rs{4}{1,2,3}[202]
			\end{scope}
		\end{scope}
		\begin{scope}[shift={(stsWt)}, scale=0.25, rface]
			\begin{scope}[scale=0.8]
				\rs{4}{6,7,8}[202]
			\end{scope}
		\end{scope}
		\begin{scope}[shift={(tstWs)}, scale=0.25, rface]
			\begin{scope}[scale=0.8]
				\rs{4}{5,6,7}[202]
			\end{scope}
		\end{scope}
		\begin{scope}[shift={(Wt)}, scale=0.25, rface]
			\begin{scope}[scale=0.8]
				\rs{4}{2,3,4}[202]
			\end{scope}
		\end{scope}
	\end{tikzpicture}
}
	\caption{The sets~$\rel(F)$ for the faces~$F$ of the cube in type~$A_2$~(left) and~$B_2$ (right). Note that positive roots point downwards.}
	\label{fig:facesCube}
\end{figure}
\end{remark}

These $\Phi$-posets are characterized as follows.

\begin{proposition}
\label{prop:characterizationBOIP}
A $\Phi$-poset~$\rel \in \IPos(\Phi)$ is in~$\BOIP(\Phi)$ if and only if ${\alpha + \beta \in \rel \implies \alpha \in \rel \text{ and } \beta \in \rel}$ for all~$\alpha, \beta \in \Phi^+$ and all~$\alpha, \beta \in \Phi^-$.
\end{proposition}

\begin{proof}
Consider first~$\rel(A,A') \in \BOIP(\Phi)$ and~$\alpha + \beta \in \rel(A,A')$ with $\alpha, \beta \in \Phi^-$.
For~${\gamma \in \Delta}$, denote by~$[\alpha:\gamma]$ the coefficient of~$\gamma$ in the decomposition of~$\alpha$ on the simple root basis.
If~$[\alpha:\gamma] \ne 0$, then~$[\alpha+\beta:\gamma] \ne 0$ which implies that~$\gamma \in A$ since ${\alpha+\beta \in \rel(A,A')^- = \rel(A)^- \subseteq \N(-A)}$.
We therefore obtain that~$\alpha \in \Phi \cap \N(-A) = \rel(A)^- \subseteq \rel(A,A')$.
By symmetry, we conclude that~$\alpha \in \rel(A,A')$ and~$\beta \in \rel(A,A')$ for any~$\alpha, \beta \in \Phi^-$ such that~$\alpha+\beta \in \rel(A,A')$.
The proof is similar when~$\alpha, \beta \in \Phi^+$.

Conversely, consider~$\rel \in \IPos(\Phi)$ such that~${\alpha + \beta \in \rel \implies \alpha \in \rel \text{ and } \beta \in \rel}$ for all~${\alpha, \beta \in \Phi^+}$ and all~$\alpha, \beta \in \Phi^-$.
Define~$A \eqdef -(\rel \cap -\Delta)$ and~$A' \eqdef \Phi \ssm (\rel \cap \Delta)$.
We claim that~${\rel = \rel(A,A')}$, \ie that~$\rel^- = \rel(A)^-$ and~$\rel^+ = \rel(A')^+$.
We prove the latter, the former is similar.
Observe first that~$\Delta \ssm A' \subseteq \rel$, so that $\rel(A')^+ = \Phi \cap \N(\Delta \ssm A') \subseteq \rel$ since~$\rel$ is closed.
Conversely, we prove by induction on~$|\gamma|$ that any~$\gamma \in \rel^+$ belongs to~$\rel(A')^+$.
Consider~$\gamma \in \rel^+$, and let~$\rel[X]$ be the multiset of simple roots such that~$\gamma = \Sigma\rel[X]$.
By Theorem~\ref{thm:filtrationSummableSubsets}, there exists~$\alpha \in \rel[X]$ such that~$\beta = \Sigma(\rel[X] \ssm \{\alpha\}) \in \Phi$.
Since~$\alpha + \beta = \gamma \in {\rel}$, we get that~$\alpha \in {\rel}$ and~$\beta \in {\rel}$.
We have~$\alpha \in \Delta \cap \rel = \Phi \ssm A' \subseteq \rel(A')^+$ and ${\beta \in \rel(A')^+}$ (by induction hypothesis).
Since~$\rel(A')^+$ is closed, this shows~${\gamma = \alpha + \beta \in \rel(A')^+}$.
\end{proof}

We are now in position to provide the proof of Proposition~\ref{prop:characterizationBOEP} postponed in Section~\ref{subsubsec:BOEP}.

\begin{proof}[Proof of Proposition~\ref{prop:characterizationBOEP}]
Observe first that for~$A \subseteq \Delta$, the boolean order element poset~$\rel(A)$ satisfies~(i) by Proposition~\ref{prop:characterizationBOIP} and (ii) since~$\alpha \in \rel(A)$ if~$\alpha \in \Delta \ssm A$ and~$-\alpha \in \rel(A)$ if~$\alpha \in A$.

Conversely, consider a $\Phi$-poset~$\rel$ satisfying~(i) and~(ii). The proof of Proposition~\ref{prop:characterizationBOIP} ensures that $\rel = \rel(A,A')$ where $A \eqdef -(\rel \cap -\Delta)$ and~$A' \eqdef \Phi \ssm (\rel \cap \Delta)$.
Condition~(ii) ensures that~$A = A'$ so that~$\rel = \rel(A,A) = \rel(A) \in \BOEP(\Phi)$.
\end{proof}

The following statement characterizes the weak order induced by~$\BOIP(\Phi)$.

\begin{proposition}
\label{prop:weakOrderBOIP}
For two boolean intervals~$A \subseteq A'$ and~$B \subseteq B'$, we have~${\rel(A,A') \wole \rel(B,B')}$ in the weak order on~$\BOIP(\Phi)$ if and only if~$A \subseteq B$ and~$A' \subseteq B'$ in boolean order.
\end{proposition}

\begin{proof}
Using Remark~\ref{rem:weakOrderBOEP}, we obtain that
\[
\begin{array}[b]{c@{\quad\iff\quad}c@{\quad\text{and}\quad}c}
\rel(A,A') \wole \rel(B,B')
& \rel(A,A')^+ \supseteq \rel(B,B')^+ & \rel(A,A')^- \subseteq \rel(B,B')^- \\
& \rel(A')^+ \supseteq \rel(B')^+ & \rel(A)^- \subseteq \rel(B)^- \\
& \Delta \ssm A' \supseteq \Delta \ssm B' & A \subseteq B \\
& A' \subseteq B' & A \subseteq B.
\end{array}
\qedhere
\]
\end{proof}

\begin{corollary}
\label{coro:latticeBOIP}
The weak order on~$\BOIP(\Phi)$ is a lattice with meet and join
\[
\rel(A,A') \meetBOIP \rel(B,B') = \rel(A \cap B, A' \cap B')
\quad\text{and}\quad
\rel(A,A') \joinBOIP \rel(B,B') = \rel(A \cup B, A' \cup B').
\]
\end{corollary}

We conclude with a connection between the lattice structure of the weak orders on~$\BOIP(\Phi)$ with that on~$\IPos(\Phi)$, $\WOIP(\Phi)$ and~$\COIP(c)$.

\begin{proposition}
\label{prop:sublatticeBOIP}
The set~$\BOIP(\Phi)$ induces a sublattice of the weak order on~$\IPos(\Phi)$, on~$\WOIP(\Phi)$ and on~$\COIP(c)$ for all Coxeter element~$c$.
\end{proposition}

\begin{proof}
\enlargethispage{.5cm}
Let~$\rel[R] = \rel(A,A')$ and~$\rel[S] = \rel(B,B')$ be two boolean order interval posets, and consider ${\rel[M] = \rel[R] \meetSC \rel[S]}$.
Observe that
\begin{align*}
\rel[M]^- & = \rel^- \cap \rel[S]^- = -\cl{A} \cap -\cl{B} = - \cl{(A \cap B)} \\
\text{and}\qquad
\rel[M]^+ & = \cl{( \rel^+ \cup \rel[S]^+ )} = \cl{\big( \cl{(\Delta \ssm A')} \cup \cl{(\Delta \ssm B')} \big)} = \cl{\big( \Delta \ssm (A' \cap B') \big)}.
\end{align*}
In other words, we obtain that~$\rel[M] = \rel[R] \meetBOIP \rel[S]$ is already in~$\BOIP(\Phi)$, and consequently
\[
\rel[R] \meetC \rel[S] = \ncd{\rel[M]} = \rel[M] = \rel[R] \meetBOIP \rel[S] \in \BOIP(\Phi).
\]
As~$\BOIP(\Phi) \subseteq \COIP(c) \subseteq \WOIP(\Phi) \subseteq \IPos(\Phi)$, we have~$\rel[R] \meetBOIP \rel[S] \wole \rel[R] \meetCOIP \rel[S] \wole \rel[R] \meetWOIP \rel[S] \wole \rel[R] \meetBOIP \rel[S]$ so that all these meets coincide.
The proof is similar for the join.
\end{proof}


\bibliographystyle{alpha}
\bibliography{latticeClosedSetsRoots}
\label{sec:biblio}

\end{document}